\documentclass{article}

\usepackage[utf8]{inputenc}
\usepackage[T1]{fontenc}
\usepackage[french,english]{babel}
\usepackage{amsmath}
\usepackage{amssymb}
\usepackage{amsthm}
\usepackage{enumitem}
\usepackage{mathtools}
\usepackage{hyperref}
\usepackage{geometry}
\usepackage{xparse}
\usepackage{tikz-cd}
\usepackage{mathrsfs}
\usepackage{csquotes}
\usepackage{verbatim}
\usepackage[backend=biber,style=alphabetic]{biblatex}
\DeclareFieldFormat{postnote}{#1} 

\renewcommand{\theparagraph}{\arabic{section}.\arabic{subsection}.\arabic{paragraph}}
\newtheorem{introthm}{Theorem}

\newtheorem{introcor}[introthm]{Corollary}

\newtheorem{thm}{Theorem}[section]
\newtheorem{lem}[thm]{Lemma}

\newtheorem{cor}[thm]{Corollary}
\newtheorem{prop}[thm]{Proposition}
\newtheorem{defi}[thm]{Definition}
\newtheorem*{defi*}{Definition}

\theoremstyle{definition}

\theoremstyle{remark}
\newtheorem*{rmk}{Remark}

\numberwithin{equation}{section}

\newcommand*{\sheafhom}{\mathscr{H}\kern -1.5pt om}

\newcommand{\RR}{\mathbb{R}}
\newcommand{\ZZ}{\mathbb{Z}}

\newcommand{\FF}{\mathbb{F}}
\newcommand{\QQ}{\mathbb{Q}}
\newcommand{\GG}{\mathbb{G}}
\newcommand{\CC}{\mathbb{C}}
\DeclareMathOperator{\Spec}{Spec}
\DeclareMathOperator{\Hom}{Hom}

\DeclareMathOperator{\Ext}{Ext}
\DeclareMathOperator{\Func}{Fun}
\DeclareMathOperator{\fib}{fib}
\DeclareMathOperator{\cofib}{cofib}
\DeclareMathOperator{\Gal}{Gal}

\tikzset{%
    symbol/.style={%
        draw=none,
        every to/.append style={%
            edge node={node [sloped, allow upside down, auto=false]{$#1$}}}
    }
} 

\title{Special values of $L$-functions on regular arithmetic schemes of dimension $1$}
\author{Adrien Morin}
\date{}

\begin{document}

\maketitle
\begin{abstract}
    We construct a well-behaved Weil-étale complex for a large class of $\ZZ$-constructible sheaves on a regular irreducible scheme $U$ of finite type over $\ZZ$ and of dimension $1$. We then give a formula for the special value at $s=0$ of the $L$-function associated to any $\ZZ$-constructible sheaf on $U$ in terms of Euler characteristics of Weil-étale cohomology; for smooth proper curves, we obtain the formula of \cite{Geisser20}. We deduce a special value formula for Artin $L$-functions twisted by a singular irreducible scheme $X$ of finite type over $\ZZ$ and of dimension $1$. This generalizes and improves all results in \cite{Tran16}; as a special case, we obtain a special value formula for the arithmetic zeta function of $X$.
\end{abstract}

\tableofcontents

\section{Introduction}

In \cite{Lichtenbaum09}, Lichtenbaum conjectured the existence of a Weil-étale topology for arithmetic schemes and that the Euler characteristic of the corresponding cohomology with compact support, constructed \textit{via} a determinant, should give special values of zeta functions at $s=0$. Lichtenbaum then gave a tentative definition of the cohomology groups in the case of the spectrum $\Spec(\mathcal{O}_K)$ of the ring of integers in a number field and showed that the conjecture holds if the cohomology groups vanish in degree $>3$; unfortunately, they were shown to be non-vanishing in \cite{Flach08}. 

In \cite{Morin14} and \cite{Flach18}, Flach and B. Morin used a new approach by defining a Weil-étale complex in the derived category of abelian groups, as the fiber of a morphism constructed \textit{via} Artin-Verdier duality; they also introduced higher regulator pairings. On the other hand, in \cite{Tran16}, Tran constructed a Weil-étale cohomology with compact support for a certain class of $\ZZ$-constructible sheaves on the spectrum $U=\Spec(\mathcal{O}_{K,S})$ of the ring of $S$-integers in a number field $K$ and used it to prove a formula for the special values of Artin L-functions associated to integral representations of the absolute Galois group of $K$. Other relevant works include Chiu's thesis \cite{Chiu} and Geisser and Suzuki's recent paper \cite{Geisser20}. 

In this paper, we improve and generalize all results of \cite{Tran16} and treat also smooth curves over finite fields as in \cite{Geisser20}, using the approach of Flach and Morin: we construct a Weil-étale complex for a larger class of $\ZZ$-constructible sheaves, and deduce a Euler characteristic defined for an arbitrary $\ZZ$-constructible sheaf with good functoriality properties. This enables us to prove a formula for the special value of the $L$-function of a $\ZZ$-constructible sheaf at $s=0$. We deduce a special value formula for Artin $L$-functions twisted by a singular irreducible scheme $X$ of finite type over $\ZZ$ of dimension $1$; as a special case, we obtain a special value formula for the arithmetic zeta function of $X$.

Let $U$ be a regular irreducible scheme of finite type over $\ZZ$ and of dimension $1$, and denote by $K$ its function field. We recall in section \ref{etale_cohomology_compact_support} the compactly supported cohomology $R\Gamma_c(U,-)$ (resp. $R\hat{\Gamma}_c(U,-)$), defined as in \cite{ADT} by correcting étale cohomology with Galois cohomology (resp. Tate cohomology) at the places in $S$ and we denote $H^i_c(U,-)$ (resp. $\hat{H}^i_c(U,-)$) the corresponding cohomology groups. A $\ZZ$-constructible sheaf $F$ on $U$ will be called big if $\Ext^1_U(F,\GG_m)$ is finite and tiny if $H^1_c(U,F)$ is finite; a morphism $F\to G$ is big-to-tiny if $F$ and $G$ are big, $F$ is big and $G$ is tiny or both $F$ and $G$ are tiny. A short exact sequence $0\to F\to G\to H\to 0$ will be called big-to-tiny if all morphisms are big-to-tiny. We prove in sections \ref{functoriality_weil_etale_complex} and \ref{functoriality_euler_characteristic}:
\begin{introthm}
For every big or tiny sheaf $F$, there exists a Weil-étale complex with compact support $R\Gamma_{W,c}(U,F)$, well-defined up to unique isomorphism. It sits in a distinguished triangle
\[
R\Hom(R\Hom_U(F,\GG_m),\QQ[-3]) \to R\Gamma_c(U,F)\to R\Gamma_{W,c}(U,F) \to
\]
It is a perfect complex, functorial in big-to-tiny morphisms, and it yields a long exact cohomology sequence for big-to-tiny short exact sequences. If $\pi:U'\to U$ is the normalization in a finite Galois extension of $K$ and $j:V\subset U$ is an open immersion, we have canonical isomorphisms $R\Gamma_{W,c}(U,\pi_\ast F)\simeq R\Gamma_{W,c}(U',F)$ and $R\Gamma_{W,c}(U,j_!F)\simeq R\Gamma_{W,c}(V,F)$.
\end{introthm}
This distinguished triangle was inspired by \cite{Morin14} : for $U$ proper and $F=\ZZ$ this is exactly the distinguished triangle in B. Morin's article. Unfortunately, the defects of the derived ($1$-)category prevented us from obtaining a well-defined complex for every $\ZZ$-constructible sheaf, and we only got a long exact cohomology sequence instead of a distinguished triangle. Next, following \cite{Flach18}, we construct in section \ref{regulator_pairing} a regulator pairing for a complex $F$ in the derived category of étale sheaves $ D(U_{et})$ that generalizes the logarithmic embedding of the Dirichlet $S$-units theorem, and we show that it is perfect after base change to $\RR$:
\begin{introthm}
There is a regulator pairing
\[
R\Gamma_c(U,F)\otimes R\Hom_U(F,\GG_m)\to\RR[-1],
\]
functorial in $F$. The induced map after base change to $\RR$
\[
R\Hom_U(F,\GG_m)_\RR \to R\Hom(R\Gamma_c(U,F),\RR[-1])
\]
is a natural isomorphism for $F\in D(U_{et})$.
\end{introthm}
The method we use to prove the theorem was inspired by \cite{Tran16}; it relies on dévissage of $\ZZ$-constructible sheaves \textit{via} Artin induction at the generic point. This method was also used in \cite{Geisser20} to treat the case of curves over a finite field.

We show in section \ref{splitting_weil_etale_complex} that the perfectness of the regulator pairing gives a canonical trivialization of the $\RR$-determinant of the Weil-étale complex
\[
\lambda :\det_\RR(R\Gamma_{W,c}(U,F)\otimes \RR) \xrightarrow{\simeq} \RR
\]
The $\ZZ$-determinant $\det_\ZZ(R\Gamma_{W,c}(U,F)$ embedds in $\det_\RR(R\Gamma_{W,c}(U,F)\otimes \RR)$ and we define:
\begin{defi*}
\begin{itemize}
\item Let $F$ be a big or tiny sheaf. The (Weil-étale) Euler characteristic of $F$
is the positive integer $\chi_U(F)$ such that \[\lambda(\det_\ZZ(R\Gamma_{W,c}(U,F))=\frac 1 {\chi_U(F)}\ZZ\hookrightarrow \RR\]
The rank of $F$ is
\[
E_U(F)=\sum_i (-1)^i\cdot i\cdot \dim_\RR(H^i_{W,c}(U,F)\otimes\RR)
\]
where $H^i_{W,c}(U,F)$ denotes hypercohomology of $R\Gamma_{W,c}(U,F)$.
\item Let $F$ be a $\ZZ$-constructible. There exists a short exact sequence $0 \to F' \to F \to F''\to 0$ with $F'$ big and $F''$ tiny; define
\[
\chi_U(F)=\chi_U(F')\chi_U(F'')
\]
It does not depend on the chosen sequence.
Proceed similarly to define $E_U(F)$.
\end{itemize}
\end{defi*}
For a $\ZZ$-constructible sheaf $F$ on $U$, define the $L$-function associated to $F$ by the usual Euler product with local factors given by the $F_v\otimes \CC$, where $F_v$ is the pullback of $F$ at a closed point $v\in U$.
%
Denote $r_U(F):=\mathrm{ord}_{s=0}L_U(F,s)$ the vanishing order of $L_U(F,s)$ at $s=0$ and 
\[L^\ast_U(F,0)=\lim_{s\to 0}L_U(F,s)s^{-r_U(F)}
\]
the special value of $L_U(F,s)$ at $s=0$.
Our main theorem then is:
\begin{introthm}[Special value formula]\label{thmC}
Let $F$ be a $\ZZ$-constructible sheaf. Then $r_U(F) = E_U(F)$ and
\[
    L^\ast_U(F,0)=(-1)^{\mathrm{rank}_\ZZ F(K)}\chi_U(F)
\]
\end{introthm}
We prove the theorem again by dévissage, once we have shown that the Euler characteristic inherits the functoriality properties of the Weil-étale complex. This theorem shows in particular that the Weil-étale Euler characteristic of a constructible sheaf is $1$.

Let $K$ be a global field. As a corollary, we obtain in section \ref{twisted_L_functions} a special value formula at $s=0$ for "twisted" Artin $L$-functions: 
\begin{defi*}
	Let $X$ be an irreducible scheme of finite type over $\ZZ$, of dimension $1$ and with residue field $K$ at the generic point, and let $M$ be a discrete $G_K$-module of finite type\footnote{as an abelian group}. The Artin L-function of $M$ (twisted by $X$) is defined as 
	\[
	L_{K,X}(M,s):=\prod_{x\in X_{0}} \det(\mathrm{Id}-N(x)^{-s}\mathrm{Frob}_x|(g_\ast M)_x \otimes \CC)^{-1}
	\]
	where $g:\Spec(K)\to X$ is the canonical morphism, $X_{0}$ is the set of closed point, $N(x)=[\kappa(x)]$ is the cardinality of the residue field at $x$, $(-)_x$ denotes pullback to $x$ and $\mathrm{Frob}_x \in \mathrm{Gal}(\kappa(x)^{sep}/\kappa(x))$ is the (arithmetic) Frobenius at $x$. Define $r_{K,X}(M)$ the order at $s=0$ and $L_{K,X}^\ast(M,0)$ the special value at $s=0$.
\end{defi*}
Theorem $C$ immediately implies:
\begin{introcor}
Let $f : X \to T$ denote either a quasi-finite morphism $X\to T:=\mathbb{P}^1_{\FF_p}$ (in the function field case), which exists by the projective Noether normalization lemma, or the quasi-finite structural morphism $X\to T:=\Spec(\ZZ)$ (in the number field case); then $f_!$ is well-defined. We have 
$r_{K,X}(M)=E_T(f_! g_\ast M)$ and
\[
 L_{K,X}^\ast(M,0)=(-1)^{\mathrm{rank}_\ZZ M}\chi_T(f_!g_\ast M)
\] 
 \end{introcor}
The Weil-étale Euler characteristic is computed explictly and doesn't depend on $f$; this is done through the use of Deninger's complex $\GG_X$ from \cite{Deninger87}, which identifies with (a shift of) the cycle complex $\ZZ^c_X(0)$.

We deduce from this a special value formula for the arithmetic zeta function of $X$:
\begin{introcor}
Suppose $X$ is affine. We have $\mathrm{ord}_{s=0}\zeta_X = \mathrm{rank}_\ZZ(CH_0(X,1))$ and
\[
	\zeta_X^\ast(0) = - \frac{[\mathrm{CH}_0(X)]R_X}{\omega}
\]
\end{introcor}
\begin{introcor}
Suppose $X$ is a proper curve, and denote $\FF_q$ the field of constants of $K$. We have $\mathrm{ord}_{s=0}\zeta_X  = \mathrm{rank}_\ZZ(CH_0(X,1))-1$ and
\[
	\zeta_X^\ast(0) = - \frac{[\mathrm{CH}_0(X)_{tor}]R_X}{\omega \log(q)}
\]
\end{introcor}
Here $\omega$ is the number of roots of unity in the field of functions, $R_X$ is a regulator and $\mathrm{CH}_0(X,i)$ denote higher Chow groups. This gives a generalization of the classical special value formula of the zeta function of a number field at $s=0$. The result seems to be new; we note that a similar formula at $s=1$ for $X$ furthermore affine and reduced was established by Jordan and Poonen in \cite{Poonen20}.

\paragraph*{Acknowledgments}
This paper was written during my doctoral studies at the Institut de Mathématiques de Bordeaux. I would like to thank my advisor Baptiste Morin; without his guidance and moral support this paper would not exist. I would also like to thank Philippe Cassou-Noguès, Thomas Geisser, Stephen Lichtenbaum, Qing Liu, Matthias Flach and Niranjan Ramachandran for helpful remarks and discussions.

\section{\'Etale cohomology with compact support}\label{etale_cohomology_compact_support}

We will be working with the derived $\infty$-category of abelian groups $\mathscr{D}(\ZZ)$; this is a stable $\infty$-category whose homotopy category is the ordinary derived category $D(\ZZ)$.\footnote{For the theory of derived and stable $\infty$-categories, see \cite[chap. 1]{HA}} Moreover, the projective model structure on unbounded chain complexes of abelian groups is a symmetric monoidal model structure, hence by \cite[4.1.7.6]{HA} the underlying $\infty$-category $\mathscr{D}(\ZZ)$ is closed symmetric monoidal (see also \cite[A.7]{Nikolaus18}), with inner hom $R\Hom(-,-)$ and tensor product $-\otimes^L-$, computed as usual. If $\mathcal{C}$ is an $\infty$-category, then the $\infty$-category of functors $\mathrm{Fun}(\mathcal{C},\mathscr{D}(\ZZ))$ is again a stable $\infty$-category. We will denote $\fib$, $\cofib$ and $[1]$ the fiber, cofiber and shift functors.

Let $U$ be a regular irreducible scheme of finite type over $\ZZ$ and of dimension $1$. We denote $K$ its function field. It is a global field; either $K$ is a number field and $U=\Spec(\mathcal{O}_{K,S})$ is the spectrum of the ring of $S$-integers for $S$ a finite set of places of $K$ including the set $S_\infty$ of all archimedean places, or $K$ is a function field of characteristic $p>0$ and $U$ is an open subscheme of the unique connected smooth proper curve $C$ with function field $K$, in which case we denote by $S$ the set of missing points\footnote{see \cite[3.1]{Poonen20}}. If $S\neq \emptyset$ then $U$ is affine \cite{Goodman69}. 

We recall the Milne étale cohomology with compact support from \cite{ADT}. For each place $v$ of $K$, let $K_v$ be the completion of $K$ at $v$ if $v$ is archimedean and the field of fraction of the henselization $\mathcal{O}^h_v$ otherwise. We will denote pullback to $\eta_v:=\Spec(K_v)$ (via the composite $\eta_v\to \Spec(K) \to U$) by $(-)_{\eta_v}$. Denote $R\Gamma(K_v,-)$ the \'etale cohomology on $\Spec(K_v)$. We have a natural morphism 
\[
R\Gamma(U,-)\to R\Gamma(K_v,(-)_{\eta_v})
\]
Now denote $R\hat{\Gamma}(K_v,-)$ for the Tate group cohomology of $\mathrm{Gal}(K_v^{sep}/K_v)$ when $v$ is archimedean, and ordinary cohomology of profinite groups otherwise. We have a natural morphism 
\[R\Gamma(K_v,-) \to R\hat{\Gamma}(K_v,-)\] for any place $v$. We then define
\begin{itemize}
    \item The (ordinary) étale complex with compact support
    \[
    R\Gamma_c(U,-):=\fib(R\Gamma(U,-)\to \prod_{v \in S}R\Gamma(K_v,(-)_{\eta_v}))
    \]
    \item The (Tate) étale complex with compact support
    \[
    R\hat{\Gamma}_c(U,-):=\fib(R\Gamma(U,-)\to \prod_{v \in S}R\hat{\Gamma}(K_v,(-)_{\eta_v}))
    \]
\end{itemize}
where we take the fiber in the $\infty$-category $\Func(\mathrm{Sh}(U_{et}),\mathscr{D}(\ZZ))$ of functors from the category of abelian sheaves on $U_{et}$ to the derived $\infty$-category of abelian groups $\mathscr{D}(\ZZ)$. The corresponding cohomology groups will be denoted $H^i_c(U,-)$ and $\hat{H}^i_c(U,-)$.

\begin{rmk}
\begin{itemize}
	\item[]
	\item Only the Tate étale complex with compact support is introduced in \cite{ADT} but we will use both crucially.
	\item If $U$ is a curve, $R\Gamma_c(U,-)$ is the usual cohomology with compact support $R\Gamma(C,j_!(-))$ (see proposition \ref{milne_II.2.3}.d) and $R\Gamma_c(U,-)=R\hat{\Gamma}_c(U,-)$.
\end{itemize}
\end{rmk}

The various functoriality properties of the étale complex with compact support are investigated in \cite{ADT}. Let us recall them (for instance for the ordinary complex with compact support) :

\begin{prop}[{\cite[II.2.3]{ADT}}]\label{milne_II.2.3}
    \begin{enumerate}[label=\alph*.]
        \item For any sheaf $F$ on $U$, there is a long exact sequence
        \begin{equation}\label{long_sequence_compact_support}
            \cdots \to H^i_c(U,F) \to H^i(U,F) \to \prod_{v\in S} H^i(K_v,F_{\eta_v}) \to H^{i+1}_c(U,F)\to \cdots
        \end{equation}
        
        \item A short exact sequence $0\to F\to G\to H$ on $U$ gives a fiber sequence
        \[R\Gamma_c(U,F)\to R\Gamma_c(U,G)\to R\Gamma_c(U,H)\]
        \item For any closed immersion $i:Z\hookrightarrow U$ and sheaf $F$ on $Z$, we have a natural isomorphism
        \[
            R\Gamma_c(U,i_\ast(-)) \simeq R\Gamma(Z,-)
        \]
        
        \item For any open immersion $j:V\to U$, we have a natural isomorphism
        \[
        R\Gamma_c(U,j_!(-))\simeq R\Gamma_c(V,-)
        \]
        Thus for any sheaf $F$ on $U$, there is a fiber sequence
        \[
        R\Gamma_c(V,F_{|V})\to R\Gamma_c(U,F)\to \bigoplus_{v\in U\backslash V}R\Gamma(v,i_v^\ast F)
        \]
        where $i_v:v\to U$ is the inclusion of a closed point.
        
        \item For any finite map $\pi:U'\to U$ with $U'$ regular irreducible, we have a natural isomorphism
        \[
        R\Gamma_c(U,\pi_\ast(-))\simeq R\Gamma_c(U',-)
        \]
    \end{enumerate}
\end{prop}

\begin{proof}
See \cite{ADT}. We have to say something about b. ; it follows from the following lemma.
\end{proof}

\begin{lem}[$3\times 3$ Lemma]\label{nine_stable}
    Let $\mathcal{C}$ be a stable $\infty$-category and let $A\to B \to C$, $A''\to B'' \to C''$ be two fiber sequences in $\mathcal{C}$. Consider a commutative diagram
    \[\begin{tikzcd}
    A \rar \dar & B \rar \dar & C \dar\\
    A'' \rar & B'' \rar & C'' 
    \end{tikzcd}\]
   Let $A'$, $B'$, $C'$ denote the respective fibers of $A\to A''$, etc. Then there is a commutative diagram 
    \[\begin{tikzcd}
    A' \rar \dar & B' \rar \dar & C' \dar\\
    A \rar \dar & B \rar \dar & C \dar\\
     A'' \rar & B'' \rar & C'' 
    \end{tikzcd}\]
    and $A'\to B' \to C'$ is a fiber sequence.
\end{lem}

\begin{proof}
By \cite[1.1.1.8]{HA}, the fiber functor $\mathrm{fib}$ preserves all limits, hence preserves fiber sequences.
\end{proof}

Recall that a $\ZZ$-constructible sheaf on $U$ is an abelian sheaf $F$ on $U_{et}$ such that $U$ can we written as the union of locally closed subschemes $Z_i$, with $F_{|Z_i}$ locally constant associated to an abelian group of finite type. Equivalently, because $U$ is of dimension $1$, this means that there is an open subscheme $V\subset U$ with closed complement $Z$, such that $F_{|V}$ is locally constant associated to an abelian group of finite type, and on the other hand for every closed point $v\in Z$, the pullback $F_v$ is a discrete $G_v:=\Gal(\kappa(v)^{sep}/\kappa(v))$-module of finite type, where $\kappa(v)$ is the residue field at $v$.

\paragraph{Dévissage for $\ZZ$-constructible sheaves} 

Let us discuss the dévissage argument we will often use ; this methodology comes from \cite{Tran16} and was used in \cite{Geisser20}. Let $F$ be a $\ZZ$-constructible sheaf. There is a short exact sequence
\[
0\to F_{tor} \to F \to F/F_{tor} \to 0
\]
where $F_{tor}$ is torsion hence constructible, and $F/F_{tor}$ is torsion free. We will often be able to handle constructible sheaves ; let us thus suppose that $F$ is torsion-free. By Artin induction (\cite[4.1,4.4]{Swan60}), there exists for the pullback $F_\eta$ to the generic point a short exact sequence of the form
\[
0\to (F_\eta)^n\oplus \bigoplus \mathrm{ind}_{G_K}^{G_{L_i}}\ZZ \to \bigoplus \mathrm{ind}^{G_ {L_j'}}_{G_K}\ZZ \to N \to 0
\]
where $n$ is an integer, $N$ is a finite $G_K$-module and $L_i$, $L_j'$ are a finite number of finite Galois extensions of $K$. By spreading out, there exists a dense open subscheme $V\subset U$, finite normalization morphims of $V$ in the above finite Galois extensions of $K$ of the form $\pi_i: W_i \to V$, $\pi_j': W_j'\to V$ and a constructible sheaf $G$ on $V$ such that there is a short exact sequence
\[
0 \to (F_{|V})^n\oplus \bigoplus (\pi_i)_\ast \ZZ \to \bigoplus (\pi_j')_\ast \ZZ \to G \to 0
\]
Finally, for $Z$ the closed complement of $V$, there is a short exact sequence
\[
0 \to F_V \to F \to F_Z \to 0
\]
where $F_V=j_! (F_{|V})$ (with $j:V\hookrightarrow U$ the open inclusion) and $F_Z=i_\ast i^\ast F = \bigoplus_{v\in Z} (i_v)_\ast i_v^\ast F$ (with $i:Z\hookrightarrow U$ resp. $i_v:v\hookrightarrow U$ the closed inclusion). We can again use Artin induction on each point of $Z$.

With this dévissage, general arguments about $\ZZ$-constructible sheaves that are functorial with respect to the operations $j_!$, $i_\ast$, $\pi_\ast$ mentioned above reduce to statements about constructible sheaves and the constant sheaf $\ZZ$ on $U$ and on a point. Once we have handled the case of a  point, the short exact sequence $0 \to j_! \ZZ \to \ZZ \to i_\ast \ZZ \to 0$ will enable us to further suppose that $U$ is either affine or proper, depending on which situation we prefer.

\paragraph{The comparison complex}

Let $v$ be an archimedean place of $K$. The complex $R\Gamma(K_v,-)$ is represented by the complex of homogeneous cochains $C^\ast(K_v,-)$, obtained by taking $\Hom$ from the standard free resolution of $\ZZ$ into a given module. Moreover, the complex $R\hat{\Gamma}(K_v,-)$ is represented by the complex $\hat{C}^\ast(K_v,-)$ obtained by taking $\Hom$ from the complete standard resolution of $\ZZ$. There is a natural arrow $C^\ast(K_v,-)\to \hat{C}^\ast(K_v,-)$. Let us introduce the "homology complex" $C_\ast(K_v,-)$, a complex in negative degree obtained by tensoring a module with the standard free resolution of $\ZZ$. The norm map $N:M\to M$ for any discrete $G_{K_v}$-module $M$ induces an arrow $C_\ast(K_v,-)\to C^\ast (K_v,-)$, and the sequence
\[
C_\ast(K_v,-)\to C^\ast (K_v,-)\to \hat{C}^\ast(K_v,-)
\]
is a fiber sequence.
Let us introduce a "comparison complex"
\[
    T:=\fib\left(\prod_{v\in S_\infty}R\Gamma(K_v,(-)_{\eta_v})\to \prod_{v\in S_\infty}R\hat{\Gamma}(K_v,(-)_{\eta_v})\right)=\prod_{v \in S_\infty}C_\ast(K_v,(-)_{\eta_v})
\]
Thus $T$ computes homology at the archimedean places.

Applying the $3\times 3$ lemma (lemma \ref{nine_stable}), we find a commutative diagram of fiber sequences
\[
\begin{tikzcd}
    \fib(R\Gamma_c(U,-)\to R\hat{\Gamma}_c(U,-)) \rar \dar & 0 \rar \dar & T \dar \\
    R\Gamma_c(U,-) \rar \dar & R\Gamma(U,-) \rar \dar["="] & \displaystyle\prod_{v \in S}R\Gamma(K_v,(-)_{\eta_v}) \dar\\
    R\hat{\Gamma}_c(U,-) \rar & R\Gamma(U,-) \rar & \displaystyle\prod_{v \in S}R\hat{\Gamma}(K_v,(-)_{\eta_v}) 
\end{tikzcd}
\]
We deduce an isomorphism $\fib\left(R\Gamma_c(U,-)\to R\hat{\Gamma}_c(U,-)\right)=T[-1]$, that is
\[
\cofib\left(R\Gamma_c(U,-)\to R\hat{\Gamma}_c(U,-)\right)=T
\]
We have
\begin{equation}\label{equation_T}
	H^i(T(F))=\left \{ \begin{array}{cc}
	\prod_{v\in S_\infty}\hat{H}^{i-1}(K_v,F_{\eta_v}) & i<0\\
	\prod_{v \in S_\infty}H_0(K_v,F_{\eta_v}) & i=0\\
	0 & i>0
	\end{array}\right.
\end{equation}

\begin{prop}\label{compact_equals_tate_compact}
    Let $F$ be a $\ZZ$-constructible sheaf on $U$. Then $H^i_c(U,F)=\hat{H}^i_c(U,F)$ for $i\geq 2$ and $H^1_c(U,F)\to \hat{H}^1_c(U,F)$ is surjective.
\end{prop}

\begin{proof}
The long exact sequence of cohomology associated to the fiber sequence $R\Gamma_c(U,F)\to R\hat{\Gamma}_c(U,F)\to T(F)$ gives the claim.
\end{proof}

\begin{prop}\label{compact_support_finite_type}
    Let $F$ be a $\ZZ$-constructible sheaf on $U$. Then $H^i_c(U,F)$ and $\hat{H}^i_c(U,F)$ are of finite type for $i=0,1$.  
\end{prop}

\begin{proof}
The cofiber $T(F)$ of $R\Gamma_c(U,F)\to R\hat{\Gamma}_c(U,F)$ computes homology at the archimedean places, thus its cohomology groups are finite type. We reduce to showing the statement only for $\hat{H}^i_c(U,F)$. By the dévissage argument and the functoriality properties of proposition \ref{milne_II.2.3}, we have to treat three cases: the constructible case, the case of the constant sheaf $\ZZ$, and the case of Galois cohomology of a discrete $G_v$-module of finite type for $v$ a closed point.
\begin{itemize}
    \item The first is part of the statement of Artin-Verdier duality (see Theorem \ref{AVduality}),
    \item The second is a computation using the long exact sequence \ref{long_sequence_compact_support}: we have $H^0(U,\ZZ)=\ZZ$ and $H^1(U,\ZZ)=\Hom_{cont}(\pi_1^{SGA3}(U),\ZZ)=0$ because $U$ is normal and noetherian hence the étale fundamental group $\pi_1^{SGA3}(U)=\pi_1^{SGA1}(U)$ is profinite. Then \ref{long_sequence_compact_support} gives an exact sequence 
    \[
    0 \to \hat{H}^0_c(U,\ZZ)\to \ZZ \to \bigoplus_{v \in S\backslash S_\infty} \ZZ \oplus \bigoplus_{v \text{~real}} \ZZ/2\ZZ \to \hat{H}^1_c(U,\ZZ) \to 0
    \]
    \item Let $G$ be a profinite group and $M$ a discrete $G$-module of finite type (as an abelian group). It is clear that $H^0(G,M)=M^G$ is also of finite type. There exists an open subgroup $H$ such that $H$ acts trivially on $G$ ; then the inflation-restriction exact sequence reduces us to verifying that $H^1(G,M)$ is of finite type for $M$ a discrete $G$-module of finite type with trivial action. Here we have $G=G_v=\hat{\ZZ}$ ; but $H^1(\hat{\ZZ},\ZZ)=\Hom_{cont}(\hat{\ZZ},\ZZ)=0$ and $H^1(\hat{\ZZ},\ZZ/n\ZZ)=\Hom_{cont}(\hat{\ZZ},\ZZ/n\ZZ)\simeq \ZZ/n\ZZ$ so we are done.
\end{itemize}
\end{proof}

\section{The Weil-étale complex with compact support}
\subsection{Construction of the Weil-étale complex with compact support}\label{construction_weil_etale}

\begin{defi}\label{def_big_tiny}
Let $F$ be a $\ZZ$-constructible sheaf on $U$. We adapt terminology from \cite{Tran16} and say that
\begin{itemize}
    \item $F$ is a big sheaf if $\Ext^1_U(F,\GG_m)$ is torsion (hence finite) \footnote{Those are, up to some modifications, the strongly $\ZZ$-constructible sheaves of \cite{Tran16} in the number field case}
    \item $F$ is a tiny sheaf if $H^1_c(U,F)$ is torsion (hence finite)
    \item A big-to-tiny morphism is a morphism of sheaves $F\to G$ where either $F$ and $G$ are both big, or are both tiny, or $F$ is big and $G$ is tiny ; a big-to-tiny short exact sequence is a short exact sequence with big-to-tiny morphisms.
\end{itemize}
\end{defi}

\begin{rmk}
\begin{itemize}
	\item[]
    \item A constructible sheaf is both big and tiny (see Theorem \ref{AVduality})
    \item A constant sheaf defined by a finite type abelian group is big, except when $U=C$ is a proper curve. Indeed, this reduces to the case of $\ZZ$, and $\Ext^1_U(\ZZ,\GG_m)=H^1(U,\GG_m)=\mathrm{Pic}(U)$ is finite when $U$ is affine\footnote{This is well-known : see for instance \cite[2.3]{Moret89}}
    \item Let $F$ be a locally constant sheaf defined by a finite type abelian group. Let $\pi:V\to U$ be a finite Galois cover with group $G$ trivializing $F$ ; the low degree exact sequence of the spectral sequence
    \[
    H^p(G,\Ext^q_V(F_{|V},\GG_m))\Rightarrow \Ext^{p+q}_U(F,\GG_m)
    \]
    shows from the constant case that $F$ is big, except when $U=C$.
    \item A $\ZZ$-constructible sheaf supported on a finite closed subscheme is tiny ; indeed, if $i:v\hookrightarrow U$ is the inclusion of a closed point, we have that $H^1_c(U,i_\ast M)=H^1(G_v,M)$ is finite by the argument given in the proof of Proposition \ref{compact_support_finite_type}.
    \item On $U=C$, a locally constant sheaf defined by a finite type abelian group is tiny. This reduces to the case of $\ZZ$, and we have $H^1_c(C,\ZZ)=H^1(C,\ZZ)=\Hom_{cont}(\pi_1(U),\ZZ)=0$.
\end{itemize}
\end{rmk}

Let $F$ be a big or tiny sheaf. We construct in this section a complex $R\Gamma_{W,c}(U,F)$ in the derived category of the integers $D(\ZZ)$, exact and functorial in big-to-tiny morphisms. We encountered some difficulty and failed to define it for all $\ZZ$-constructible sheaves, but this will be sufficient to obtain a well-defined Euler characteristic. Following the approach of \cite{Morin14} and \cite{Flach18}, the complex fits in a distinguished triangle
\[
R\Hom(R\Hom_U(F,\GG_m),\QQ[-3])\to R\Gamma_c(U,F)\to R\Gamma_{W,c}(U,F)\to
\]
where the map $R\Hom(R\Hom_U(F,\GG_m),\QQ[-3])\to R\Gamma_c(U,F)$ is lifted from the canonical map $R\Hom(R\Hom_U(F,\GG_m),\QQ[-3])\to R\Hom(R\Hom_U(F,\GG_m),\QQ/\ZZ[-3])$ \textit{via} Artin-Verdier duality.

Denote
\begin{align*}
    D:=R\Hom(R\Hom_U(-,\GG_m),\QQ[-3])\\
    E:=R\Hom(R\Hom_U(-,\GG_m),\QQ/\ZZ[-3])
\end{align*}
Artin-Verdier duality gives a pairing $R\hat{\Gamma}_c(U,-)\otimes R\Hom_U(-,\GG_m) \to \QQ/\ZZ[-3]$ that induces an "Artin-Verdier" arrow 
\[
R\hat{\Gamma}_c(U,-)\to E.\]
We have $H^i(D)=\Ext^{3-i}_U(-,\GG_m)^\dagger$ and $H^i(E)=\Ext^{3-i}_U(-,\GG_m)^D$, where $(-)^\dagger:=\Hom(-,\QQ)$ and $(-)^D:=\Hom(-,\QQ/\ZZ)$. The cohomology groups of the Artin-Verdier arrow are described as following:

\begin{thm}[Artin-Verdier duality, {\cite[II.3.1]{ADT}}]\label{AVduality}
Let $F$ be a $\ZZ$-constructible sheaf on $U$. The groups $\Ext^i_U(F,\GG_m)$ and $\hat{H}^i_c(U,F)$ are finite for $i<0$, $i>4$, of finite type for $i=0,1$ and torsion of cofinite type for $i=2,3$. Moreover $\hat{H}^i_c(U,F)=0$ for $i\geq 4$. The Artin-Verdier arrow induces an isomorphism for $i=0,1$
\[
\hat{H}^i_c(U,F)^\wedge \xrightarrow{\simeq} \Ext^{3-i}_U(F,\GG_m)^D
\]
were $(-)^\wedge$ is the profinite completion with respect to subgroups of finite index,
and for $i\neq 0,1$ an isomorphism
\[
\hat{H}^i_c(U,F) \xrightarrow{\simeq} \Ext^{3-i}_U(F,\GG_m)^D
\]
If $F$ is constructible, all groups are finite and the Artin-Verdier arrow is an isomorphism.
\end{thm}

We recall that an abelian group is torsion of cofinite type if it is of the form $(\QQ/\ZZ)^n\oplus A$ where $A$ is finite and $n$ an integer.

The theorem implies in particular that $D$ is cohomologically concentrated in degree $2$ and $3$.

\begin{rmk}
In \cite{ADT}, the theorem is expressed in its dual form ; we obtain the one above by applying $(-)^\ast:=\Hom_{cont}(-,\RR/\ZZ)$, which coincides with $\Hom_{cont}(-,\QQ/\ZZ)$ for profinite groups (where $\QQ/\ZZ$ has the discrete topology) and with $(-)^D$ for discrete torsion groups. We use furthermore that for $M$ discrete of finite type, $(M^\wedge)^\ast=M^\ast$ and the Pontryagin duality $M=M^{\ast\ast}$ for any abelian locally compact $M$.

The fact that $\hat{H}^4_c(U,F)=0$ does not appear in \cite{ADT} but is easily deduced by the dévissage arguement from the functoriality properties of $\hat{H}^i_c(U,-)$, the case of $\ZZ$, which is \cite[II.2.11(d)]{ADT}, the case of constructible sheaves, which is \cite[II.3.12(a)]{ADT}, and the case of Galois cohomology of $\hat{\ZZ}$, which is of strict cohomological dimension $2$ (see for instance \cite[XIII.1]{Serrelocaux}).
\end{rmk}

We want to construct an arrow $\beta_F$ in $D(U_{et})$ making the following diagram commute
\[
    \begin{tikzcd}
    D_F \rar["\beta_F",dotted] \arrow[drr,"\pi",swap] & R\Gamma_c(U,F) \rar & R\hat{\Gamma}_c(U,F) \dar["\text{Artin-Verdier}"]\\
    && E_F
    \end{tikzcd}
\]
Let us analyze $\Hom_{D(\ZZ)}(D_F,R\Gamma_c(U,G))$ for $F$ big or $G$ tiny. There is a spectral sequence for $\Hom_{D(\ZZ)}(K,L)$ (\cite[III.4.6.10]{theseverdier}) of the form
\[
E_2^{p,q}=\prod_{i \in \ZZ}\mathrm{Ext}^p(H^i(K),H^{i+q}(L))\Rightarrow H^n(R\Hom(K,L))=\Hom(K,L[n])=E^n
\]
In $D(\ZZ)$ we have $\Ext^i=0$ for $i\neq 0,1$ so the sequence degenerates and gives a short exact sequence
\begin{equation}\label{verdier_exact_sequence}
    0 \to \prod_i \Ext^1(H^i(K),H^{i-1}(L)) \to \Hom_{D(\ZZ)}(K,L) \to \prod_i \Hom(H^i(K),H^i(L)) \to 0
\end{equation}
In our case, this gives
\begin{align*}
0 &\to \Ext^1(\Ext^1_U(F,\GG_m)^\dagger,H^1_c(U,G))\times \Ext^1(\Hom_U(F,\GG_m)^\dagger,H^2_c(U,G)) \to \Hom_{D(\ZZ)}(D_F,R\Gamma_c(U,G))\\ &\to \prod_i\Hom(H^i(D_F),H^i(R\Gamma_c(U,G))) \to 0
\end{align*}
where $(-)^\dagger:=\Hom(-,\QQ)$. Since $H^2_c(U,-)$ is torsion of cofinite type and either $\Ext^1_U(F,\GG_m)$ or $H^1_c(U,G)$ are finite, the two terms in the product on the left vanish and we have an isomorphism
\[
\Hom_{D(\ZZ)}(D_F,R\Gamma_c(U,G)) \xrightarrow{\simeq} \Hom(H^2(D_F),H^2_c(U,G))\times \Hom(H^3(D_F),H^3_c(U,G))
\]
because $D_F$ is cohomologically concentrated in degree $2$ and $3$. The Artin-Verdier arrow is an isomorphism in degree $2$ and $3$, and $\hat{H}^i_c(U,F)=H^i_c(U,F)$ for $i\geq 2$, so the arrow $\beta_F$ exists and is unique. Moreover a big-to-tiny morphism $f:F\to G$ gives rise to a commutative square
\[
\begin{tikzcd}
D_F \dar["f_\ast"] \rar["\beta_F"] & R\Gamma_c(U,F) \dar["f_\ast"]\\
D_G \rar["\beta_G"] & R\Gamma_c(U,G)
\end{tikzcd}
\]
Indeed by the above isomorphism it suffices to check commutativity on cohomology groups in degree $2$ and $3$.

\begin{defi}
Let $F$ be either big or tiny. We define
\[
R\Gamma_{W,c}(U,F):=Cone(\beta_F)
\]
the Weil-étale cohomology with compact support with coefficients in $F$.
\end{defi}

\subsection{Computing the Weil-étale cohomology with compact support}\label{computations}

\begin{defi}
Let $F$ be a big or tiny sheaf. We denote by $H^i_{W,c}(U,F):=H^i(R\Gamma_{W,c}(U,F))$ the Weil-étale cohomology groups with compact support.
\end{defi}
The complex $D_F$ is cohomologically concentrated in degree $2$ and $3$, so the long exact cohomology sequence gives $H^i_{W,c}(U,F)=H^i_c(U,F)$ for $i\leq 0$, $i\geq 4$. In particular $H^i_{W,c}(U,F)=0$ for $i\neq 0,1,2,3$. The remaining long exact sequence reads 
\[\begin{tikzcd}[sep=small]
& 0 \rar & H^1_c(U,F) \rar & H^1_{W,c}(U,F) &\\
\rar & \Ext^1(F,\GG_m)^\dagger \rar["nat"] & H^2_c(U,F) \rar & H^2_{W,c}(U,F)&\\
\rar & \Hom(F,\GG_m)^\dagger \rar["nat"] & H^3_c(U,F) \rar & H^3_{W,c}(U,F) \rar & 0
\end{tikzcd}\]
From the construction, we see that the natural arrows $\Ext^{3-i}_U(F,\GG_m)^\dagger\to H^{i}_c(U,F)\simeq \hat{H}^{i}_c(U,F)\simeq \Ext_U^{3-i}(F,\GG_m)^D$ for $i=0,1$ identify with the post-composition with the projection $\QQ\to \QQ/\ZZ$. We thus get short exact sequences and an isomorphism
\begin{align*}
    & 0\to H^1_c(U,F) \to H^1_{W,c}(U,F)\to \Hom(\Ext^1_U(F,\GG_m),\ZZ)\to 0\\
    & 0 \to (\Ext^1_U(F,\GG_m)_{tor})^D \to H^2_{W,c}(U,F) \to \Hom(\Hom_U(F,\GG_m),\ZZ) \to 0\\
    & H^3_{W,c}(U,F)\simeq (\Hom_U(F,\GG_m)_{tor})^D
\end{align*}
In particular, all cohomology groups are of finite type, hence $R\Gamma_{W,c}(U,F)$ is a perfect complex.

\subsection{Functoriality of the Weil-étale complex}\label{functoriality_weil_etale_complex}

\begin{thm}\label{long_exact_cohomolgy_sequence}
Let $F$ be either big or tiny. The complex $R\Gamma_{W,c}(U,F)$ is well-defined up to unique isomorphism, is functorial in big-to-tiny morphisms, and gives a long exact cohomology sequence for big-to-tiny short exact sequences.
\end{thm}
\begin{proof}
Let us fix a choice of a cone of $\beta_F$ for every big or tiny sheaf $F$. Let $f:F\to G$ be a big-to-tiny morphism. We consider the beginning of a morphism distinguished triangle 
\[
\begin{tikzcd}
D_F \rar["\beta_F"] \dar["f_\ast"] & R\Gamma_c(U,F) \dar["f_\ast"] \rar & R\Gamma_{W,c}(U,F) \rar & D_F[1] \dar["f_\ast"] \\
D_G \rar["\beta_G"] & R\Gamma_c(U,G) \rar & R\Gamma_{W,c}(U,G) \rar & D_G[1]
\end{tikzcd}
\]
The groups $H^i(D_F)$ are $\QQ$-vector spaces and zero for $i\neq 2,3$ and the Weil-étale cohomology groups with compact support are of finite type and zero for $i\neq 0,1,2,3$. The exact sequence \eqref{verdier_exact_sequence} for $\Hom_{D(\ZZ)}$ gives
\begin{align*}
     0 &\to \prod_i \Ext^1(H^i(D_F),H^{i}_{W,c}(U,F)) \to \Hom_{D(\ZZ)}(D_F[1],R\Gamma_{W,c}(U,F))\\ &\to \prod_i \Hom(H^{i+1}(D_F),H^i_{W,c}(U,F)) \to 0
\end{align*}
\begin{align*}
     0 &\to \prod_i \Ext^1(H^i_{W,c}(U,F),H^{i-1}_{W,c}(U,G)) \to \Hom_{D(\ZZ)}(R\Gamma_{W,c}(U,F),R\Gamma_{W,c}(U,G))\\ &\to \prod_i \Hom(H^i_{W,c}(U,F),H^i_{W,c}(U,G)) \to 0
\end{align*}
so that $\Hom_{D(\ZZ)}(D_F[1],R\Gamma_{W,c}(U,G))$ is uniquely divisible and  $\Hom_{D(\ZZ)}(R\Gamma_{W,c}(U,F),R\Gamma_{W,c}(U,G))$ is of finite type. In the long exact sequence obtained by applying $\Hom_{D(\ZZ)}(-,R\Gamma_{W,c}(U,G))$ to the upper triangle, the image of $\Hom_{D(\ZZ)}(D_F[1],R\Gamma_{W,c}(U,G))$ in $\Hom_{D(\ZZ)}(R\Gamma_{W,c}(U,F),R\Gamma_{W,c}(U,G))$ must be uniquely divisible and of finite type, hence is trivial and we get an injection
\[
0 \to \Hom_{D(\ZZ)}(R\Gamma_{W,c}(U,F),R\Gamma_{W,c}(U,G)) \to \Hom_{D(\ZZ)}(R\Gamma_{c}(U,F),R\Gamma_{W,c}(U,G))
\]
Therefore a morphism completing the morphism of triangles, which exists by axiom, is uniquely determined by the requirement that it makes the square
\[
\begin{tikzcd}
R\Gamma_c(U,F) \dar["f_\ast"] \rar & R\Gamma_{W,c}(U,F) \dar[dotted] \\
R\Gamma_c(U,G) \rar & R\Gamma_{W,c}(U,G)
\end{tikzcd}
\]
commute. Thus there is a unique morphism $f_\ast:R\Gamma_{W,c}(U,F)\to R\Gamma_{W,c}(U,G)$ completing the diagram of distinguished triangles above, which gives at the same time the uniqueness and desired functoriality of $R\Gamma_{W,c}(U,-)$ (for the uniqueness, consider the case of two different choices of cone of $\beta_F$ and $\mathrm{id}: F\to F$).

Let us show the exactness. Let $0\to F \to G \to H \to 0$ be a big-to-tiny short exact sequence. We obtain a commutative diagram, where the two left columns and all rows are distinguished triangles.
\begin{equation}\label{exactitude_weil_etale}
   \begin{tikzcd}
D_F \rar["\beta_F"] \dar & R\Gamma_c(U,F) \rar \dar & R\Gamma_{W,c}(U,F) \rar \dar["\exists !"] & D_F[1] \dar\\
D_G \rar["\beta_G"] \dar & R\Gamma_c(U,G) \rar \dar & R\Gamma_{W,c}(U,G) \rar \dar["\exists !"] & D_G[1] \dar\\
D_H \rar["\beta_H"] \dar & R\Gamma_c(U,H) \rar \dar & R\Gamma_{W,c}(U,H) \rar \dar["\exists !"] & D_H[1] \dar\\
D_F[1] \rar["{\beta_F[1]}"]  & R\Gamma_c(U,F)[1] \rar  & R\Gamma_{W,c}(U,F)[1] \rar & D_F[2]
\end{tikzcd} 
\end{equation}

Actually we still have to construct the expected arrow $R\Gamma_{W,c}(U,H)\to R\Gamma_{W,c}(U,F)[1]$. We reason as before: the exact sequence \eqref{verdier_exact_sequence} gives
\begin{align*}
    0 \to & \Ext^1(\Ext^1_U(H,\GG_m)^\dagger,H^2_c(U,F))\times \Ext^1(\Hom_U(H,\GG_m)^\dagger,H^3_c(U,F)) \to
\Hom_{D(\ZZ)}(D_H,R\Gamma_c(U,F)[1])\\ & \to \prod_i \Hom(H^i(D_H),H^{i+1}_c(U,F)) \to 0
\end{align*}

Since $H^i_c(U,F)=\hat{H}^i_c(U,F)$ is torsion of cofinite type for $i\geq 2$ and zero for $i\geq 4$, we readily see that the left term is $0$, hence a canonical isomorphism
\[
\Hom_{D(\ZZ)}(D_H,R\Gamma_c(U,F)[1])\xrightarrow[H^2]{\simeq} \Hom(H^2(D_H),H^{3}_c(U,F))
\]
Thus in diagram \eqref{exactitude_weil_etale}, we can check that the bottom left square commutes on cohomology groups in degree $2$ ; but then the corresponding square is
\[
\begin{tikzcd}
\Ext^1_U(H,\GG_m)^\dagger \rar \dar & H^2_c(U,H)=\Ext^1_U(H,\GG_m)^D \dar\\
\Hom_U(F,\GG_m)^\dagger \rar & H^3_c(U,F)=\Hom_U(H,\GG_m)^D
\end{tikzcd}\]
which clearly commutes. Now the existence and uniqueness of the wanted arrow folllows by the same methods as for the others.

Let us now take the cohomology groups in the third row; this gives a long sequence
\[ \cdots \to H^i_{W,c}(U,F)\to H^i_{W,c}(U,G) \to H^i_{W,c}(U,H) \to H^{i+1}_{W,c}(U,F)\to \cdots\]
and we have to check that it is exact. We will use results from \cite[\href{https://stacks.math.columbia.edu/tag/0BKH}{Tag 0BKH}]{stacks-project} regarding derived completions. We have $D_F\otimes^L \ZZ/n\ZZ=0$. This gives an isomorphism
\[
R\Gamma_c(U,F)\otimes^L \ZZ/n\ZZ \simeq R\Gamma_{W,c}(U,F)\otimes^L \ZZ/n\ZZ
\]
hence by taking derived limits (\cite[\href{https://stacks.math.columbia.edu/tag/0922}{Tag 0922}]{stacks-project}) an isomorphism
\[
R\Gamma_c(U,F)^\wedge_p \simeq R\Gamma_{W,c}(U,F)^\wedge_p\simeq R\Gamma_{W,c}(U,F)\otimes \ZZ_p
\]
between the derived $p$-completions ; the last isomorphism holds because $R\Gamma_{W,c}(U,F)$ is a perfect complex and $\ZZ_p$ is flat (\cite[\href{https://stacks.math.columbia.edu/tag/0EEV}{Tag 0EEV}]{stacks-project}). Derived $p$-completion is an exact functor (\cite[\href{https://stacks.math.columbia.edu/tag/091V}{Tag 091V}]{stacks-project}) so the functor $R\Gamma_c(U,-)^\wedge_p$ is exact. Now we can just check that our long sequence is exact after tensoring with $\ZZ_p$ for each prime $p$, which follows from the exactness of $R\Gamma_{W,c}(U,-)\otimes \ZZ_p$.
\end{proof}
\section{The regulator pairing}\label{regulator_pairing}

\subsection{Construction}

We denote by $\otimes$ the derived tensor product in $\mathscr{D}(\ZZ)$ and $A_\RR:=A\otimes \RR$ for any abelian group or complex of abelian groups.

Taking inspiration from \cite[2.2]{Flach18}, we define in this section a pairing on the derived $\infty$-category $\mathscr{D}(U_{et})$ of étale sheaves on $U$
\[
R\Gamma_c(U,-)\otimes R\Hom_U(-,\GG_m)\to \RR[-1]
\]
with values in the derived $\infty$-category $\mathscr{D}(\ZZ)$, and we show that it is perfect after base change to $\RR$.  The pairing is functorial, meaning that the induced morphism
\[
R\Hom_U(-,\GG_m)\to R\Hom(R\Gamma_c(U,-),\RR[-1])
\]
is a natural transformation. When $U$ is a curve over $\FF_q$, we will be able to define it with value in $\QQ[-1]$\footnote{Though it will depend on $q$; the $\RR$-valued regulator doesn't depend on $q$}, and its base change to $\QQ$ will be perfect.

The derived $\infty$-category $\mathscr{D}(U_{et})$ is naturally enriched in the symmetric monoidal $\infty$-category $\mathscr{D}(\ZZ)$ via the complexes $R\Hom_U(-,-)$ ; hence for $F,G\in \mathscr{D}(U_{et})$ we have the $\Ext$ pairings\footnote{coming from the Yoneda pairing, noting that $R\Gamma(U,F)=R\Hom_U(\ZZ,F)$}
\[
   R\Gamma(U,F)\otimes R\Hom_U(F,G)\to R\Gamma(U,G)
\]
\[
    \prod_{v\in S}R\Gamma(K_v,F_{\eta_v})\otimes R\Hom_U(F,G) \to \prod_{v\in S}R\Gamma(K_v,F_{\eta_v})\otimes \prod_{v\in S}R\Hom_{K_v}(F_{\eta_v},G_{\eta_v})\to \prod_{v\in S} R\Gamma(K_v,G_{\eta_v}) 
\]
Since the derived tensor product is exact, we have a commutative diagram
\[\begin{tikzcd}
    R\Gamma_c(U,F)\otimes R\Hom_U(F,G) \rar & R\Gamma(U,F)\otimes R\Hom_U(F,G) \rar \dar & (\displaystyle\prod_{ v\in S}R\Gamma(K_v,F_{\eta_v}))\otimes R\Hom_U(F,G) \dar \\
    R\Gamma_c(U,G) \rar & R\Gamma(U,G) \rar & \displaystyle\prod_{v\in S}R\Gamma(K_v,G_{\eta_v})
\end{tikzcd}\]
where both rows are fiber sequences, hence an induced morphism
\[
R\Gamma_c(U,F)\otimes R\Hom_U(F,G)\to R\Gamma_c(U,G).
\]

Specialise for $G=\GG_m[0]$ ; we now wish to construct a map $R\Gamma_c(U,\GG_m)\to\RR[-1]$, resp. a factorization through $\QQ[-1]$ of that map when $U$ is a curve. We begin by computing the cohomology groups $H^i_c(U,\GG_m)$.

\begin{defi}
Define
\begin{align*}
    I_K' &= \{(\alpha_v)\in \prod_{v \text{~place de~}K} K_v^\times,~\alpha_v\in\mathcal{O}_{K_v}^\times\text{~for almost all~}v\}\\
\end{align*}
the \textit{henselian} idele group. Define the \textit{henselian} idele class group and \textit{henselian} $S$-idele class group
\begin{align*}
    C_K' &= I_K'/K^\times\\
    C_{K,S}' &= \mathrm{coker}(\prod_{v\in U_0}\mathcal{O}_{K_v}^\times\to I_K'\to C_K')
\end{align*}
where $U_0$ denotes the closed points of $U$.
\end{defi}

\begin{rmk}
	\begin{itemize}
	\item[]
		\item These differ from the usual definitions because we use for $K_v$ the fraction field of the henselization at $v$ when $v$ is non-archimedean.
		\item By the snake Lemma we have $C_{K,S}'\simeq (\bigoplus_{v\in U_0}\ZZ \oplus \prod_{v\in S}K_v^\times)/K^\times$. It follows that $C_{K,S}'=\mathrm{Pic}(C)$ when $U=C$ is a proper curve.
	\end{itemize}
\end{rmk}

\begin{prop}["Hilbert 90"]
We have $H_c^0(U,\GG_m)=0$, resp. $H^0_c(U,\GG_m)$ is finite when $U$ is affine resp. $U=C$ is a proper curve, and $H^1_c(U,\GG_m)=C_{K,S}'$.
For $i\geq 2$, \[
H^i_c(U,\GG_m)=\hat{H}^i_c(U,\GG_m)=\left\{\begin{array}{cc}
    \QQ/\ZZ & i=3 \\
    0 & i> 3
\end{array}\right.\]
\end{prop}

\begin{proof}
For $i=0$ this comes from the defining long exact sequence
\[
0 \to H^0_c(U,\GG_m)\to \mathcal{O}_U(U)^\times \to \prod_{v\in S} K_v^\times \to \cdots
\]
and for $U=C$ a proper regular curve of characteristic $p>0$, $\mathcal{O}_C(C)$ is a finite extension of $\FF_p$. We already saw in Proposition \ref{compact_equals_tate_compact} that $H^i_c(U,-)=\hat{H}^i_c(U,-)$ for $i>1$, thus for $i>1$ see \cite[III.2.6]{ADT}.

Let us treat the case $i=1$. Define a complex $R\Gamma_{zar,c}$ similarly to $R\Gamma_c$, but using Zariski cohomology on $U$ and $\Spec(K_v)$ instead of étale cohomology. Let $\varepsilon$ be the morphism of sites $U_{zar}\to U_{et}$. We denote by $\varepsilon_\ast:F\mapsto F\circ \varepsilon$ the direct image functor and $\varepsilon^\ast$ its left adjoint. We have $\varepsilon_\ast (\GG_m)_{U_{et}}= (\GG_m)_{U_{zar}}$, thus $\tau^{\leq 0} R\varepsilon_\ast \GG_m=\GG_m$ and we get a morphism
\[
R\Gamma_{zar}(U,\GG_m)=R\Gamma_{zar}(U,\tau^{\leq 0} R\varepsilon_\ast\GG_m)\to R\Gamma_{zar}(U,R\varepsilon_\ast\GG_m)=R\Gamma(U,\GG_m).
\]
This combined with the similar result on $\Spec(K_v)$ induce a canonical morphism
\[
R\Gamma_{zar,c}(U,\GG_m)\to R\Gamma_c(U,\GG_m)
\]
This induces a commutative diagram in cohomology
\[
\begin{tikzcd}
0 \rar & \mathcal{O}_U(U)^\times \rar \dar[equal] & \prod_{v\in S} K_{v}^{\times} \rar \dar[equal] & H^1_{zar,c}(U,\GG_m) \rar \dar & H^1_{zar}(U,\GG_m)=\mathrm{Pic}(U) \rar \dar[equal] & 0\\
0 \rar & \mathcal{O}_U(U)^\times \rar & \prod_{v\in S}K_v^\times \rar & H^1_{c}(U,\GG_m) \rar & H^1(U,\GG_m)=\mathrm{Pic}(U) \rar & 0
\end{tikzcd}
\]
where the identifications on the right follows by Hilbert's theorem 90. Thus by the five lemma
\[
H^1_{zar,c}(U,\GG_m)\xrightarrow{\simeq}H^1_c(U,\GG_m)
\]
On the Zariski site of $U$, we have a short exact sequence
\[
0 \to \GG_m \to g_\ast \GG_m \to \bigoplus_{v\in U_0}(i_v)_\ast\ZZ \to 0
\]
where $g:\Spec(K)\hookrightarrow U$ is the inclusion of the generic point and $i_v$ is the inclusion of a closed point $v\in U$. This gives a distinguished triangle
\[
R\Gamma_{zar,c}(U,\GG_m) \to R\Gamma_{zar,c}(U,g_\ast \GG_m) \to \bigoplus_{v\in U_0}\ZZ \to
\]
For $q>0$, the sheaf $R^q g_\ast F$ is the sheaf associated to
\[
V\mapsto H^q(V\times_U\Spec(K),F)=H^q(\Spec(K),F)=0
\]
hence $g_\ast =Rg_\ast$ is exact and
\[
R\Gamma_{zar}(U,g_\ast \GG_m)=R\Gamma_{zar}(\Spec(K),\GG_m)=K^\times[0].
\]
Thus there is an identification
\[
R\Gamma_{zar,c}(U,g_\ast\GG_m)\simeq \mathrm{fib}\left(K^\times[0] \to \prod_{v\in S}K_v^\times[0]\right)=\left[K^\times \to \prod_{v\in S}K_v^\times\right]
\]
with $K^\times$ in degree zero. We find
\[
R\Gamma_{zar,c}(U,\GG_m)\simeq \mathrm{fib}\left(\left[K^\times \to \prod_{v\in S}K_v^\times\right] \to \bigoplus_{v\in U_0}\ZZ[0]\right)=\left[K^\times \to \bigoplus_{v\in U_0}\ZZ\oplus \prod_{v\in S}K_v^\times\right]
\]
hence the result.
\end{proof}

The map $\sum_v \log|\cdot|_v$ is well-defined on the henselian idele group $I_K'$ by the product formula and induces a map on the henselian $S$-idele class group $tr : C_{K,S}' \to \RR$. For $v\notin S$, we have in particular on the factor $\ZZ$ associated to $v$:
\begin{equation}\label{tr_sur_Z}
    tr(1_v)=-\log(N(v))
\end{equation}
where $N(v)=[\kappa(v)]$ is the cardinality of the residue field at $v$; this follows from the conventions chosen for the product formula. If $U$ is a curve over $\FF_q$, then $tr$ factors through $\ZZ$ as $C_{K,S}' \xrightarrow{tr_q} \ZZ \xrightarrow{\log(q)} \RR$ with $tr_q$ the map induced by $\sum_v \log_q|\cdot|_v$. Then for $v\notin S$, we have $tr_q(1_v)=-[\kappa(v):\FF_q]$. In particular, if $U=C$ is a proper curve, then $tr_q=\mathrm{deg}$ is the degree map $\mathrm{Pic}(C)\to \ZZ$.

We can consider the composite
\[
R\Gamma_c(U,\GG_m)\to \tau^{\geq 1} R\Gamma_c(U,\GG_m)_{\RR} = H^1_c(U,\GG_m)_{\RR}[-1]=(C_{K,S}')_\RR[-1]\xrightarrow{tr}\RR[-1]
\]
When $U$ is a curve over $\FF_q$, we can also consider the composite
\[
R\Gamma_c(U,\GG_m)\to \tau^{\geq 1} R\Gamma_c(U,\GG_m)_{\QQ} = H^1_c(U,\GG_m)_{\QQ}[-1]=(C_{K,S}')_\QQ[-1]\xrightarrow{tr_q\otimes\QQ}\QQ[-1],
\]
whose composite with $\QQ \xrightarrow{1\mapsto \log(q)} \RR$ is the previous map\footnote{Note that we have to tensor with $\QQ$ to supress the torsion coming from $H^3_c(U,\GG_m)=\QQ/\ZZ$.}

\begin{defi}
The regulator pairing is the pairing
\[
R\Gamma_c(U,F)\otimes R\Hom_U(F,\GG_m)\to R\Gamma_c(U,\GG_m)\to\RR[-1]
\]
induced by $tr$.

If $U$ is a curve over $\FF_q$, the rational regulator pairing is the pairing
\[
R\Gamma_c(U,F)\otimes R\Hom_U(F,\GG_m)\to R\Gamma_c(U,\GG_m)\to \QQ[-1]
\]
induced by $tr_q$.
\end{defi}
Note that the rational regulator pairing depends on the choice of the base field.

\subsection{Perfectness of the regulator pairing after base change to \texorpdfstring{$\RR$}{R}}

We call a pairing of perfect complexes $A\otimes B\to C$ perfect if the natural morphism $A\to R\Hom(B,C)$ is an isomorphism.

\begin{thm}\label{regulator_perfect}
The regulator pairing 
\[
R\Gamma_c(U,F)\otimes R\Hom_U(F,\GG_m)\to R\Gamma_c(U,\GG_m)\to\RR[-1]
\]
is perfect after base change to $\RR$ for every $\ZZ$-constructible sheaf $F$. More generally, the map
\[
R\Hom_U(F,\GG_m)_\RR \to R\Hom(R\Gamma_c(U,F)_\RR,\RR[-1])
\]
is an isomorphism for every $F\in \mathscr{D}(U_{et})$. When U is a curve over $\FF_q$, the theorem holds over $\QQ$.
\end{thm}

\begin{rmk}
\begin{itemize}
	\item[]
	\item For $U=\Spec(\mathcal{O}_K)$ and $F=\ZZ$ this is conjecture $\mathbf{B}(U,0)$ of \cite{Flach18} (note that $d=1$), which is shown in \textit{ibid.}, section 2.2 to be equivalent to Beilinson's conjecture. The latter is known for $U=\Spec(\mathcal{O}_K)$.
	\item For $F\in \mathscr{D}(U_{et})$ not $\ZZ$-constructible, the vector spaces involved may not be finite dimensional in which case the dual map $R\Gamma_c(U,F)_\RR\to R\Hom(R\Hom_U(F,\GG_m)_\RR,\RR[-1])$ will not be an isomorphism.
\end{itemize}

\end{rmk}

\begin{proof}
If $U$ is a curve over $\FF_q$, the rational regulator pairing induces the regulator pairing up to a non-zero factor $\log(q)$ ; hence for curves we will only prove the perfectness of the rational regulator pairing after base change to $\QQ$.

We have to show that $R\Hom_U(F,\GG_m)_\RR \to R\Hom(R\Gamma_c(U,F),\RR[-1])$ is an isomorphism (resp. $R\Hom_U(F,\GG_m)_\QQ \to R\Hom(R\Gamma_c(U,F),\QQ[-1])$ if $U$ is a curve). This is a natural transformation that behaves well with fiber sequences and filtered colimits, hence we first reduce to the case of bounded complexes, then by filtering with the truncations to sheaves, then to $\ZZ$-constructible sheaves. By the dévissage argument, we now have to check that
\begin{enumerate}
    \item the regulator pairing behaves well with the pushforward $\pi_\ast$ coming from a finite generically Galois morphism $\pi:V\to U$
    \item the regulator pairing behaves well with the extension by zero $j_!$ coming from an open subscheme $j:V\hookrightarrow U$
    \item the regulator pairing for a $\ZZ$-constructible sheaf supported on a closed point identifies with a pairing in Galois cohomology that is pefect after base change to $\QQ$
    \item the theorem is true for a constructible sheaf
    \item the theorem is true for the constant sheaf $\ZZ$
\end{enumerate}

\paragraph{Case of the pushforward by a finite generically Galois cover}\label{cas_loc_const}

Let $\pi :V\to U$ be the normalization of $U$ in a finite Galois extension $L/K$ and let $F$ be a $\ZZ$-constructible sheaf on $V$. Then $V=\Spec(\mathcal{O}_{L,S'})$ or $V=C'\backslash S'$ where $S'$ are the places of $L$ above those of $S$ and $C'$ is the connected smooth proper curve with function field $L$ in the function field case. Denote $g_v:\Spec(K_v)\to U$ (resp. $g_w:\Spec(L_w)\to V$) and $\pi_w : \Spec(L_w) \to \Spec(K_v)$ the natural morphism for $v\in S$, $w\in S'$. We remark that (\cite[3.6]{Tran16}) $g_v^\ast \pi_\ast =\prod_{w|v}(\pi_w)_\ast g_w^\ast$.
]
On the one hand, prop. \ref{milne_II.2.3}.e gives a natural isomorphism $R\Gamma_c(V,-)=R\Gamma_c(U,\pi_\ast -)$. On the other hand, the norm induces a map $Nm:\pi_\ast \GG_m\to \GG_m$ such that the composite 
\[
   N : R\Hom_V(F,\GG_m)\xrightarrow{\pi_\ast}R\Hom_U(\pi_\ast F,\pi_\ast \GG_m)\xrightarrow{Nm_\ast}R\Hom_U(\pi_\ast F,\GG_m)
\]
is an isomorphism (\cite[II.3.9]{ADT}). The following diagrams are commutative by functoriality of the Yoneda pairing:
\[
\begin{tikzcd}
R\Hom_V(\ZZ,F)\otimes R\Hom_V(F,\GG_m) \dar \rar & R\Gamma(V,\GG_m)=R\Hom_V(\ZZ,\GG_m) \dar\\
R\Hom_U(\ZZ,\pi_\ast F)\otimes R\Hom_U(\pi_\ast F,\pi_\ast \GG_m) \dar["\mathrm{id}\otimes Nm_\ast"] \rar & R\Gamma(U,\pi_\ast \GG_m) \dar["Nm_\ast"]\\
R\Gamma(U,\pi_\ast F)\otimes R\Hom_U(\pi_\ast F, \GG_m)\rar & R\Gamma(U,\GG_m)
\end{tikzcd}
\]
\medskip
\[
\begin{tikzcd}
R\Gamma(L_w,F_{\eta_w})\otimes R\Hom_{L_w}(F_{\eta_w},\GG_m) \dar \rar & R\Gamma(L_w,\GG_m) \dar\\
R\Gamma(K_v,(\pi_w)_\ast F_{\eta_w})\otimes R\Hom_{K_v}((\pi_w)_\ast F_{\eta_w},(\pi_w)_\ast \GG_m) \dar["\mathrm{id}\otimes (N_{L_w/K_v})_\ast"] \rar & R\Gamma(K_v,(\pi_w)_\ast \GG_m) \dar["(N_{L_w/K_v})_\ast"]\\
R\Gamma(K_v,(N_{L_w/K_v})_\ast F_{\eta_w})\otimes R\Hom_{K_v}((N_{L_w/K_v})_\ast F_{\eta_w}, \GG_m)\rar & R\Gamma(K_v,\GG_m)
\end{tikzcd}
\]
\medskip
\[
\begin{tikzcd}
R\Hom_V(F,\GG_m) \dar["Nm_\ast\circ\pi_\ast"] \rar["\prod_{w|v} j_w^\ast"] & \prod_{w|v}R\Hom_{L_w}(F_{\eta_w},\GG_m) \dar["\prod_{w|v}(N_{L_w/K_v})_\ast\circ(\pi_w)_\ast"] \\
R\Hom_U(\pi_\ast F,\GG_m) \rar["j_v^\ast"] & \prod_{w|v} R\Hom_{K_v}((\pi_w)_\ast F_{\eta_w},\GG_m)
\end{tikzcd}
\]
We obtain a diagram:
\[
\begin{tikzcd}
R\Gamma_c(V,F)\otimes R\Hom_V(F,\GG_m) \rar \dar["\simeq"] & R\Gamma_c(V,\GG_m) \rar \dar & (C_{L,S'})_\RR[-1] \dar["Nm_\ast"] \rar["tr"] & \RR[-1] \dar[equal]\\
R\Gamma_c(U,\pi_\ast F)\otimes R\Hom_U(\pi_\ast F,\GG_m) \rar & R\Gamma_c(U,\GG_m) \rar & (C_{K,S})_\RR[-1] \rar["tr"] & \RR[-1]
\end{tikzcd}
\]
The left hand side is commutative by the previous diagrams. The long exact sequence \eqref{long_sequence_compact_support} gives
\[
0 \to \mathcal{O}_U(U)^\times \to \prod_{v\in S} K_v^\times \to C_{K,S}' \to \mathrm{Pic}(U)
\]
Thus the map $tr$ factors through 
\[
C'_{K,S}\otimes \RR = \left((\prod_{v\in S}K_v^\times)/\mathcal{O}_U(U)^\times \right)\otimes \RR
\]
on which it is given by
\[
(\alpha_v) \mapsto \sum_{v\in S}\log(|\alpha_v|_v),
\] so the right hand side is commutative by the following computation:
\[
\sum_{v\in S} \log|\prod_{w|v}N_{L_w/K_v}(\alpha_w)|_v=\sum_{w\in S'} \log|N_{L_w/K_v}(\alpha_w)|_v=\sum_{w\in S'} \log|\alpha_w|_w
\]
We have proved that the theorem is true for $F$ if and only if it is true for $\pi_\ast F$ ; the reasoning is exactly the same for the rational regulator pairing when $U$ is a curve.

\paragraph{Case of the extension by zero along an open inclusion}\label{case_extension_zero}
Let $j:V\hookrightarrow U$ be an open subscheme and let $F$ be a $\ZZ$-constructible sheaf on $V$. Write $V=\Spec{\mathcal{O}_{K,S'}}$ or $V=C\backslash S'$. Using prop. \ref{milne_II.2.3}.d and the identification $j^\ast \GG_m=\GG_m$, we obtain a commutative diagram
\[
\begin{tikzcd}
R\Gamma_c(V,F)\otimes R\Hom_V(F,\GG_m) \rar \dar["\simeq"] & R\Gamma_c(V,\GG_m) \rar \dar & (C_{K,S'})_\RR[-1] \dar \rar["tr"] & \RR[-1] \dar[equal]\\
R\Gamma_c(U,j_! F)\otimes R\Hom_U(j_! F,\GG_m) \rar & R\Gamma_c(U,\GG_m) \rar & (C_{K,S})_\RR[-1] \rar["tr"] & \RR[-1]
\end{tikzcd}
\]
which shows that the theorem is true for $F$ if and only if it is true for $j_!F$. The reasoning is the same for the rational regulator pairing when $U$ is a curve.

\paragraph{Case of a sheaf supported on a closed point}\label{case_sheaf_closed_point}

Consider a $\ZZ$-constructible sheaf supported on a closed point $v$. Let $G_v=\mathrm{Gal}(\kappa(v)^{sep}/\kappa(v))$. The sheaf is of the form $i_\ast M$ where $M$ is a discrete $G_v$-module of finite type and $i:v\hookrightarrow U$ is the closed inclusion. The functor $i_\ast$ has a right adjoint in the derived category $Ri^!$ such that $i_\ast Ri^!\simeq \mathrm{id}$, and one has an identification (\cite[section 1, eq. 1.2]{Mazur73})
\begin{equation}\label{right_adjoint_closed_immersion}
Ri^!\GG_m\simeq \ZZ[-1]
\end{equation}
Prop. \ref{milne_II.2.3}.c gives a natural isomorphism $R\Gamma_c(U,i_\ast M)=R\Gamma(v,M)$ and we obtain morphisms
\begin{align*}
& R\Gamma(v,\ZZ[-1]) = R\Gamma_c(U,i_\ast \ZZ[-1])=R\Gamma_c(U,i_\ast Ri^!\GG_m)\xrightarrow{\eta_\ast}R\Gamma_c(U,\GG_m)\\
& R\Hom_{v}(M,\ZZ[-1])=R\Hom_{v}(M,Ri^!\GG_m)\xrightarrow[\simeq]{\eta_\ast\circ i_\ast}R\Hom_U(i_\ast M,\GG_m)
\end{align*}
where $\eta:i_\ast Ri^!\GG_m\to \GG_m$ is the counit of the adjunction. The functoriality of the Yoneda pairing then implies that the following diagram is commutative:
\[
\begin{tikzcd}
R\Gamma(v,M)\otimes R\Hom_v(M,\ZZ[-1]) \rar \dar["\simeq"] & R\Gamma(v,\ZZ[-1])\dar \\
R\Gamma_c(U,i_\ast M)\otimes R\Hom_U(i_\ast M,\GG_m) \rar & R\Gamma_c(U,\GG_m)
\end{tikzcd}
\]
The Galois cohomology groups of $G_v\simeq\hat{\ZZ}$ with coefficients in $\ZZ$ are given by
\[
H^i(\hat{\ZZ},\ZZ)=\left\{\begin{array}{ccr}
\ZZ & i=0\\
\Hom_{cont}(\hat{\ZZ},\ZZ)=0 &i=1\\
\Hom_{cont}(\hat{\ZZ},\QQ/\ZZ)=\QQ/\ZZ & i=2 &\\
0 & i>2
\end{array}\right.
\]
hence
\[
\tau^{\geq 1} R\Gamma(G_v,\ZZ[-1])_\QQ \simeq H^0(G_v,\ZZ)_\QQ[-1]=\QQ[-1]
\]
The map $\eta_\ast : \ZZ=H^0(G_v,\ZZ)=H^0_c(U,i_\ast \ZZ)\to H^1_c(U,\GG_m)=C_{K,S}'$ sends $\ZZ$ to the factor indexed by $v$ in $C_{K,S}'$. Following remark \ref{tr_sur_Z}, we get a commutative diagram
\begin{equation}\label{regulateur_point}
\begin{tikzcd}[cramped,sep=small]
R\Gamma(v,M)\otimes R\Hom_v(M,\ZZ[-1]) \rar \dar["\simeq"] & R\Gamma(v,\ZZ[-1])\dar \rar & \tau^{\geq 1} R\Gamma(v,\ZZ[-1])_\QQ \rar["="] \dar & \QQ[-1] \dar["-\log(N(v))"]\\
R\Gamma_c(U,i_\ast M)\otimes R\Hom_U(i_\ast M,\GG_m) \rar & R\Gamma_c(U,\GG_m) \rar &\tau^{\geq 1}R\Gamma_c(U,\GG_m)_\RR=(C_{K,S}')_\RR[-1] \rar["tr"] & \RR[-1]
\end{tikzcd}
\end{equation}

\paragraph{Case of Galois cohomology of a finite field}\label{cas_cohomologie_galoisienne_corps_fini}

Since $-\log(N(v))\neq 0$, we are reduced by the previous section to showing that
\[
R\Gamma(v,M)\otimes R\Hom_v(M,\ZZ)\to R\Gamma(v,\ZZ)\to \QQ[0]
\]
is perfect after base change to $\QQ$ for any discrete $G_v$-module $M$ of finite type as an abelian group; this will imply that it is also perfect after base change to $\RR$. We show it on cohomology groups. Galois cohomology is torsion and Artin-Verdier duality gives that $\Ext^i_v(M,\ZZ)=\Ext^{i+1}_U(i_\ast M,\GG_m)$ (equation \eqref{right_adjoint_closed_immersion}) is finite for $i\geq 1$, so the pairing is trivially perfect in degree $i\neq 0$ after base change to $\QQ$. Thus it suffices to show the result in degree $0$, i.e. to show that the pairing
\[
M^{G_v}\times \Hom_{G_v}(M,\ZZ)\to \QQ
\]
is perfect after base change to $\QQ$. Let $H$ be an open subgroup acting trivially on $M$ ; then $M^{G_v}=(M)^{G_v/H}$ and $\Hom_{G_v}(M,\ZZ)=\Hom_{G_v/H}(M,\ZZ)$, i.e. we reduce to the case of a cyclic group $G:=G_v/H$ acting of the finite type abelian group $M$. We have furthermore $\Hom_G(M,\ZZ)=\Hom(M_G,\ZZ)$ (where $M_G$ denotes coinvariants). Let $F$ denote a generator of $G$. The exact sequence
\[
0 \to M^G \to M \xrightarrow{F-1} M \to M_G\to 0,
\]
gives identifications $(M^G)_\QQ=(M_\QQ)^G$ and $(M_G)_\QQ=(M_\QQ)_G$ by tensoring with $\QQ$, and also gives that those vector spaces have the same dimension. Let $N:M_G\to M^G$ denote the norm and $\pi:M^G\to M \to M_G$ the restriction of the canonical projection to $M^G$. We notice that $N \circ \pi=|G|\mathrm{id}$, so $\pi_\QQ$ is an isomorphism. In the pairing
\[
M_\QQ^G\times \Hom(M_{G,\QQ},\QQ)\to \QQ,
\] the left hand terms have same dimension, so it suffices to prove that the left kernel is trivial, which follows immediately for surjectivity of $\pi_\QQ$.

\paragraph{Case of a constructible sheaf}\label{cas_constructible}
Let $F$ be a constructible sheaf on $U$. Then cohomology with compact support differs from Tate cohomology with compact support by finite groups. Artin-Verdier duality thus shows that all cohomology groups $H^i_c(U,F)$ and $\Ext^i_U(F,\GG_m)$ are finite, hence the pairing is trivial after base change to $\RR$ (or $\QQ$ when $U$ is a curve), so it is perfect.

\paragraph{Case of the constant sheaf $\ZZ$}\label{case_Z}

Consider the constant sheaf $\ZZ$ on $U$. If $U$ is a curve over $\FF_q$, the previous discussion along with the dévissage argument allows us to suppose that $U=C$ is a proper curve. We will show perfectness of the regulator pairing on cohomology groups. Since $U$ is noetherian and normal, $H^q(U,\ZZ)$ is torsion for $q>0$ (\cite[II.2.10]{ADT}) and $\pi_1^{SGA3}(U)=\pi_1^{SGA1}(U)$ is profinite, hence $H^1(U,\ZZ)=\Hom_{cont}(\pi_1^{et}(U),\ZZ)=0$. The same argument holds for $\ZZ$ on $\Spec(K_v)$ for $v \in S$. Thus:

\begin{itemize}
	\item If $U=\Spec(\mathcal{O}_{K,S})$, we have the exact sequence \eqref{long_sequence_compact_support}
	\[
	0\to H^0_c(U,\ZZ)\to \ZZ \to\prod_{v\in S}\ZZ \to H^1_c(U,\ZZ) \to 0
	\]
	and we obtain
	\begin{equation}\label{cohomomology_compact_Z}
	H^0_c(U,\ZZ)=0, ~~~H^1_c(U,\ZZ)=(\prod_{v\in S}\ZZ)/\ZZ, ~~~H^i_c(U,\ZZ)_\RR=0 ~\text{for}~i>1
	\end{equation}
	On the other hand
	\[
	\Ext^i_U(\ZZ,\GG_m)=H^i(U,\GG_m)=\left\{\begin{array}{cc}
	\mathcal{O}_{K,S}^\times  & i=0 \\
	\mathrm{Pic}(U) & i=1\\
	\text{torsion of cofinite type} & i\geq 2
	\end{array}\right.
	\]
	The pairing $H^i_c(U,\ZZ)_\RR\times \Ext^{1-i}(\ZZ,\GG_m)_\RR \to \RR$ is thus perfect for $i=0$ by finiteness of the $S$-ideal class group and trivially perfect for $i\neq 0,1$. In degree $i=1$, the regulator pairing identifies with
	\begin{align*}
	(\prod_{v\in S}\ZZ)/\ZZ \times \mathcal{O}_{K,S}^\times\to \RR\\
	((\alpha_v),u)\mapsto \sum \alpha_v\log|u|_v
	\end{align*}
	which is non-degenerate modulo torsion by the theorem of $S$-units of Dirichlet (see for instance \cite[3.3.12]{EATAN}).
	Indeed, under the indentification
	\[
	\Hom((\prod_{v\in S}\RR)/\RR,\RR)\simeq \{(x_i)\in \RR^s,\sum x_i=0\},\]
	the map
	\[
	\mathcal{O}_{K,S}^\times \to \Hom((\prod_{v\in S}\RR)/\RR,\RR)\subset \RR^s
	\]
	induced by the pairing is the usual logarithmic embedding.
	\item If $U=C$ is a proper curve over $\FF_q$, we have $H^0(C,\ZZ)=\ZZ$, $H^i(C,\ZZ)$ torsion for $i>0$ and \[
	\Ext^i_C(\ZZ,\GG_m)=H^i(C,\GG_m)=\left\{\begin{array}{cc}
	\mathcal{O}_C(C)^\times  & i=0 \\
	\mathrm{Pic}(C) & i=1\\
	\text{torsion of cofinite type} & i\geq 2
	\end{array}\right.
	\]
	Since $C$ is proper, normal and connected, $\mathcal{O}_C(C)=K\cap \bar{\FF}_q$ is the field of constants of $K$, and a finite extension of $\FF_q$; in particular $H^0(C,\GG_m)$ is finite. The pairing $H^i(C,\ZZ)_\QQ \times \Ext^{1-i}(\ZZ,\GG_m)_\QQ \to \QQ$ is thus trivially perfect for $i\neq 0$. In degree $i=0$, the rational regulator pairing identifies with the top row in the following commutative diagram :
	\[\begin{tikzcd}
	\QQ \times \mathrm{Pic}(C)_\QQ \rar \arrow{d}{\mathrm{id}\times\mathrm{deg}_\QQ}[swap]{\simeq} & \QQ \dar[equal]\\
	\QQ \times \QQ \arrow{r}{r,s\mapsto rs} & \QQ
	\end{tikzcd}\]
	The bottom pairing is perfect, so it remains only to show that the degree map is an isomorphism rationally, which follows from the exact sequence $0 \to \mathrm{Pic}^0(C) \to \mathrm{Pic}(C) \xrightarrow{\mathrm{deg}} f\ZZ \to 0$,\footnote{We have $f=[\mathcal{O}_C(C):\FF_q]$, see for instance \cite[VII.5, cor. 5]{Weil}} and the finiteness of the class group of degree $0$ divisors $\mathrm{Pic}^0(C)$.\footnote{See for instance \cite[IV.4 Thm 7]{Weil}; the group $k^1_\mathbb{A}/k^\times \Omega(\emptyset)$ of \textit{ibid.} identifies with $\mathrm{Pic}^0(C)$, as discussed at the beginning of \textit{ibid.}, Ch. VI}
\end{itemize}
\end{proof}
\section{Splitting of the Weil-étale complex after base change to \texorpdfstring{$\RR$}{R}}\label{splitting_weil_etale_complex}

We denote by $(-)_\RR:=-\otimes^L \RR$ the base change by $\RR$.

\begin{prop}
	Let $F$ be a big or tiny sheaf. The Weil-étale complex with compact support splits rationally, naturally in big-to-tiny morphisms and big-to-tiny short exact sequences:
	\[
		R\Gamma_{W,c}(U,F)_\QQ = R\Gamma_c(U,F)_\QQ \oplus D_{F,\QQ}[1]
	\]
	Moreover, by the perfectness of the regulator pairing after base change to $\RR$, there is an isomorphism, natural in big-to-tiny morphisms and big-to-tiny short exact sequences:
	\begin{equation}\label{splitting}
	\tau_F : R\Gamma_{W,c}(U,F)_\RR \xrightarrow{\simeq} D_{F,\RR}[1]\oplus D_{F,\RR}[2]
	\end{equation}
\end{prop}

\begin{proof}
By tensoring with $\QQ$ the defining triangle of $R\Gamma_{W,c}(U,F)$ and rotating once, we obtain a distinguished triangle
\[
R\Gamma_c(U,F)_\QQ \xrightarrow{i_F} R\Gamma_{W,c}(U,F)_\QQ \xrightarrow{p_F} D_{F,\QQ}[1] \to R\Gamma_c(U,F)_\QQ[1]
\]
Let us show that this triangle splits canonically. We will prove the existence of a section of $p$: 
\[
s:D_F[1]_\QQ \to R\Gamma_{W,c}(U,F)_\QQ.
\]

Let $G$ be another big or tiny sheaf. Consider the exact sequence obtained by applying $\Hom(D_{F,\QQ}[1],-)$ to the above triangle for $G$ :
\[
\begin{tikzpicture}[descr/.style={fill=white,inner sep=1.5pt}]
        \matrix (m) [
            matrix of math nodes,
            row sep=1em,
            column sep=2.5em,
            text height=1.5ex, text depth=0.25ex
        ]
        { \Hom(D_{F,\QQ}[1],R\Gamma_c(U,G)_\QQ) & \Hom(D_{F,\QQ}[1],R\Gamma_{W,c}(U,G)_\QQ) &  \\
            & \Hom(D_{F,\QQ}[1],D_{G,\QQ}[1]) & \Hom(D_{F,\QQ}[1],R\Gamma_c(U,G)_\QQ[1])\\
        };

        \path[overlay,->, font=\scriptsize,>=latex]
        (m-1-1) edge (m-1-2)
        (m-1-2) edge[out=355,in=175] (m-2-2)
        (m-2-2) edge (m-2-3);
        
\end{tikzpicture}
\]
Since $D_{F,\QQ}[1]$ is cohomologically concentrated in degrees $3$ and $4$ and $R\Gamma_c(U,G)_\QQ$ is concentrated in degrees $0$ and $1$, the usual argument with the exact sequence \eqref{verdier_exact_sequence} shows that both left and right terms are zero, hence a canonical isomorphism:
\[
\Hom(D_{F,\QQ}[1],R\Gamma_{W,c}(U,G)_\QQ)\xrightarrow{\simeq}\Hom(D_{F,\QQ}[1],D_{G,\QQ}[1]).
\]
The same argument also gives $\Hom(D_{F,\QQ}[1],R\Gamma_{W,c}(U,G)_\QQ[1])\xrightarrow{\simeq}\Hom(D_{F,\QQ}[1],D_{G,\QQ}[2])$.

Let $s_F$ be the inverse image of the identity under $\Hom(D_{F,\QQ}[1],R\Gamma_{W,c}(U,F)_\QQ)\xrightarrow{\simeq}\Hom(D_{F,\QQ}[1],D_{F,\QQ}[1])$. The following diagram gives a morphism of distinguished triangles (note that the right square does commute, since $\Hom(D_{F,\QQ}[1],R\Gamma_c(U,F)_\QQ[1])=0$):
\[
\begin{tikzcd}
R\Gamma_c(U,F)_\QQ \rar \dar[equal] & R\Gamma_c(U,F)_\QQ \oplus D_{F,\QQ}[1] \rar \dar["{(i_F,s_F)}"] & D_{F,\QQ}[1] \rar \dar[equal] & R\Gamma_c(U,F)_\QQ[1] \dar[equal]\\
R\Gamma_c(U,F)_\QQ \rar["i"] & R\Gamma_{W,c}(U,F)_\QQ \rar & D_{F,\QQ}[1] \rar["p"] & R\Gamma_c(U,F)_\QQ[1]\\
\end{tikzcd}
\]
hence by the five lemma an isomorphism
\[
(i_F,s_F):R\Gamma_c(U,F)_\QQ \oplus D_{F,\QQ}[1] \xrightarrow{\simeq} R\Gamma_{W,c}(U,F)_\QQ
\]

Let us see the naturality of the above isomorphism. Let $F\to G$ be a big-to-tiny morphism. We have already proved that $i$ is natural (see the proof of Theorem \ref{long_exact_cohomolgy_sequence}), so it suffices to show that $s$ is a natural transformation, i.e. that the following diagram commutes:
\[
\begin{tikzcd}
D_{F,\QQ}[1]\rar["s_H"] \dar & R\Gamma_{W,c}(U,F) \dar \\
D_{G,\QQ}[1]\rar["{s_G}"] & R\Gamma_{W,c}(U,G)
\end{tikzcd}
\]
The horizontal arrows in the following diagram are isomorphisms, and the diagram is commutative by commutativity of \eqref{exactitude_weil_etale}, whence we conclude with a simple diagram chasing:
\[
\begin{tikzcd}
\Hom(D_{F,\QQ}[1],R\Gamma_{W,c}(U,F)_\QQ)\rar["\simeq"] \dar & \Hom(D_{F,\QQ}[1],D_{F,\QQ}[1]) \dar\\
\Hom(D_{F,\QQ}[1],R\Gamma_{W,c}(U,G)_\QQ)\rar["\simeq"] & \Hom(D_{F,\QQ}[1],D_{G,\QQ}[1]) \\
\Hom(D_{G,\QQ}[1],R\Gamma_{W,c}(U,G)_\QQ)\rar["\simeq"] \uar & \Hom(D_{G,\QQ}[1],D_{G,\QQ}[1]) \uar
\end{tikzcd}
\]

Let $0\to F \to G \to H \to 0$ be a big-to-tiny short exact sequence. Similarly to the previous step, $i$ is natural, and it remains to see that the diagram
\[
\begin{tikzcd}
D_{H,\QQ}[1]\rar["s_H"] \dar & R\Gamma_{W,c}(U,H) \dar \\
D_{F,\QQ}[2]\rar["{s_F[1]}"] & R\Gamma_{W,c}(U,F)[1]
\end{tikzcd}
\]
commutes. We conclude by a similar diagram chasing in the diagram
\[
\begin{tikzcd}
\Hom(D_{H,\QQ}[1],R\Gamma_{W,c}(U,H)_\QQ)\rar["\simeq"] \dar & \Hom(D_{H,\QQ}[1],D_{H,\QQ}[1]) \dar\\
\Hom(D_{H,\QQ}[1],R\Gamma_{W,c}(U,F)_\QQ[1])\rar["\simeq"] & \Hom(D_{H,\QQ}[1],D_{F,\QQ}[2]) \\
\Hom(D_{F,\QQ}[2],R\Gamma_{W,c}(U,F)_\QQ[1])\rar["\simeq"] \uar & \Hom(D_{F,\QQ}[2],D_{F,\QQ}[2]) \uar
\end{tikzcd}.
\]

The regulator pairing gives an isomorphism, natural in complexes of étale sheaves $F\in \mathscr{D}(U)$
\begin{align*}
    R\Gamma_c(U,F)_\RR  \xrightarrow{\simeq} R\Hom((R\Hom_U(F,\GG_m),\RR[-1])&=R\Hom(R\Hom_U(F,\GG_m),\QQ[-1])\otimes \RR\\
    &= D_{F,\RR}[2]
\end{align*}
hence the claimed isomorphism natural in big-to-tiny morphisms and big-to-tiny short exact sequences
\[
\tau_F : R\Gamma_{W,c}(U,F)_\RR \xrightarrow{\simeq} D_{F,\RR}[1]\oplus D_{F,\RR}[2]
\]
\end{proof}

\begin{rmk}
	\begin{itemize}
		\item[]
		\item If $U$ is a curve the same holds over $\QQ$, using the rational regulator pairing.
		\item The above isomorphism implies that the functor $R\Gamma_{W,c}(U,-)_\RR$ is exact in big-to-tiny short exact sequences.
	\end{itemize}
\end{rmk}

\section{The Weil-étale Euler characteristic}

\subsection{Construction}\label{Construction}

In this section we will use the determinant construction of Knudsen-Mumford \cite{detdiv}, and the subsequent work of Breuning, Burns and Knudsen, in particular \cite{Breuning08}. Let $R$ be a Noetherian ring; denote $\mathrm{Proj}_R$ the exact category of projective finite type $R$-modules, $\mathrm{Gr}^b(\mathrm{Mod}_R^{ft})$ the bounded graded abelian category of finite type $R$-modules and $D_{perf}(R)$ the derived category of perfect complexes. For $\mathcal{P}$ a Picard groupoid, the Picard groupoid of determinants $\mathrm{det}(\mathcal{C},\mathcal{P})$ for $\mathcal{C}$ an exact or triangulated category is defined in \cite{Knudsen02} and \cite{Breuning11} respectively. Their objects are determinants, that is pairs $f=(f_1,f_2)$ where $f_1$ is a functor $\mathcal{C}^{iso}\to \mathcal{P}$ and $f_2$ maps each exact sequence/distinguished triangle $\Delta : X\to Y \to Z$ to an isomorphism $f_2(\Delta) : f_1(Y)\to f_1(X)\otimes f_1(Z)$, satisfying certain axioms. When the context is clear we'll abuse notation and write $f$ for both $f_1$ and $f_2$. There is a canonical inclusion $I:\mathrm{Proj}_\ZZ\hookrightarrow \mathrm{Gr}^b(\mathrm{Mod}_R^{ft})$ where $I(A)$ is $A$ in degree $0$ and $0$ elsewhere. We extend $I$ naturally to exact sequences. Thus $I$ defines a restriction functor 
\[
I^\ast:\mathrm{det}(\mathrm{Gr}^b(\mathrm{Mod}^{ft}_R),\mathcal{P}) \to \mathrm{det}(\mathrm{Proj}_\ZZ,\mathcal{P})
\]
for any Picard groupoid $\mathcal{P}$, given by $f=(f_1,f_2)\mapsto (f_1\circ I, f_2\circ I)$. On the other hand, there is a functor $H:D_{perf}(R)\to \mathrm{Gr}^b(\mathrm{Mod}^{ft}_R)$ given by $X\mapsto (H^i(X))_i$, and for a distinguished triangle 
\[
\Delta: X\xrightarrow{u} Y \xrightarrow{v} Z \xrightarrow{w}
\] we define $H(\Delta)$ to be the following exact sequence, induced by the long exact cohomology sequence:
\[
0 \to \mathrm{ker}(H(u))\to H(X) \to H(Y) \to H(Z) \xrightarrow{\partial} \mathrm{ker}(H(u))[1]\to 0
\]
Thus $H$ defines an extension functor $H^\ast: \mathrm{det}( \mathrm{Gr}^b(\mathrm{Mod}^{ft}_R),\mathcal{P}) \to \mathrm{det}( D_{perf}(R),\mathcal{P})$, mapping a determinant $(g_1,g_2)$ to $(g_1\circ H,f_2)$ where $f_2$ is deduced from applying $g_2$ to the exact sequences of the form $H(\Delta)$ for $\Delta$ a distinguished triangle \cite[Prop. 5.8]{Breuning11}.

The functor $I$ factors as $\mathrm{Proj}_\ZZ \to \mathrm{Mod}^{ft}_R \to \mathrm{Gr}^b(\mathrm{Mod}^{ft}_R)$ hence for $R$ a regular ring, $I^\ast$ is an equivalence of categories for any Picard category $\mathcal{P}$ by \cite[Prop. 5.5]{Breuning11} and Quillen's resolution theorem \cite[Cor. 2 in §4]{Quillen73} (see also the proof of \cite[Prop. 3.4]{Breuning08}). Note that by \cite[Lemma 5.10, Prop. 5.11]{Breuning11}, $H^\ast$ is also an equivalence of categories since there is an identification $D_{perf}(R)\simeq D^b(\mathrm{Mod}_R^{ft})$.

Consider the usual determinant functor 
\[
 \mathrm{det}_R \in \mathrm{det}(\mathrm{Proj}_R, \mathcal{P}_R), ~ M\mapsto (\Lambda^{top}M,\mathrm{rank}M)
\]
 of \cite{detdiv} with values in the Picard groupoid of graded $R$-lines, that is of pairs $(L,n)$ where $L$ is an invertible $R$-module and $n:\Spec(R)\to \ZZ$ a locally constant function, with morphisms $(L,n)\to (L',m)$ the isomorphisms $L\to L'$ if $n=m$ and none otherwise.
\begin{rmk}[{\cite[Section 2.5]{Burns03}}]
 We have $\pi_0\mathcal{P}_R=H^0(\Spec(R),\ZZ)\oplus \mathrm{Pic}(R)$ and $\pi_1 \mathcal{P}_R = R^\times$. Hence by $K$-theoretic computations, $\mathrm{det}_R$ is the universal determinant for $R$ local, or semi-simple, or the ring of integers in a number field.
\end{rmk}
For basic properties of determinant functors on triangulated categories, see \cite[Section 3]{Breuning11}. Let $g_R\in\mathrm{det}(\mathrm{Gr}^b(\mathrm{Mod}^{ft}_R), \mathcal{P}_R)$ be a determinant functor extending $\det_R$, i.e. such that $I^\ast(g_R)=\det_R$. Define then $f_R:=H^\ast(g_R)\in \mathrm{det}(D_{perf}(R), \mathcal{P}_R)$.

We specialize now to $R=\ZZ,~\RR$. Whenever it makes sense, denote by $B$ the base change functor $X\mapsto X\otimes_\ZZ \RR$. The following diagram is commutative
\[
\begin{tikzcd}
\mathrm{Proj}_\ZZ \arrow[r, "I", hook] \arrow[d, "B"] & \mathrm{Gr}^b(\mathrm{Mod}^{ft}_\ZZ) \arrow[d, "B"] & D_{perf}(\ZZ) \arrow[l, "H"] \arrow[d, "B"] \\
\mathrm{Proj}_\RR \arrow[r, "I", hook]                & \mathrm{Gr}^b(\mathrm{Mod}^{ft}_\RR)                & D_{perf}(\RR) \arrow[l, "H"]               
\end{tikzcd}
\]
and there is a natural symmetric monoidal functor $B:\mathcal{P}_\ZZ\to \mathcal{P}_\RR$, hence we obtain a commutative diagram
\[
\begin{tikzcd}
{\mathrm{det}(\mathrm{Proj}_\ZZ,\mathcal{P}_\ZZ)} \arrow[d, "B_\ast"] & {\mathrm{det}(\mathrm{Gr}^b(\mathrm{Mod}^{ft}_\ZZ), \mathcal{P}_\ZZ)} \arrow[l, "I^\ast"] \arrow[d, "B_\ast"] \arrow[r, "H^\ast"] & {\mathrm{det}(D_{perf}(\ZZ), \mathcal{P}_\ZZ)} \arrow[d, "B_\ast"] \\
{\mathrm{det}(\mathrm{Proj}_\ZZ,\mathcal{P}_\RR)}                     & {\mathrm{det}(\mathrm{Gr}^b(\mathrm{Mod}^{ft}_\ZZ), \mathcal{P}_\RR)} \arrow[l, "I^\ast"] \arrow[r, "H^\ast"]                      & {\mathrm{det}(D_{perf}(\ZZ), \mathcal{P}_\RR)}                     \\
{\mathrm{det}(\mathrm{Proj}_\RR,\mathcal{P}_\RR)} \arrow[u, "B^\ast"] & {\mathrm{det}(\mathrm{Gr}^b(\mathrm{Mod}^{ft}_\RR), \mathcal{P}_\RR)} \arrow[l, "I^\ast"] \arrow[u, "B^\ast"] \arrow[r, "H^\ast"] & {\mathrm{det}(D_{perf}(\RR), \mathcal{P}_\RR)} \arrow[u, "B^\ast"]
\end{tikzcd}
\]
where upper stars denote precomposition and lower stars denote postcomposition. From this diagram, we see that there is an isomorphism of determinants $\gamma : B\circ g_\ZZ \simeq B^\ast(g_\RR)$ inducing an isomorphism
\begin{equation}\label{base_change_det}
H^\ast(\gamma) : B\circ f_\ZZ \simeq B^\ast(f_\RR)
\end{equation}

Applying the determinant construction to the isomorphism \ref{splitting}, we get a trivialization
\[
\lambda_F : f_\RR(R\Gamma_{W,c}(U,F)_\RR) \simeq f_\RR(D_{F,\RR}[1]\oplus D_{F,\RR}[2]) \xleftarrow{\simeq} f_\RR(D_{F,\RR}[1]])\otimes f_\RR(D_{F,\RR}[2]) \xrightarrow{\simeq} \RR
\]
where the last isomorphism holds because for any perfect complex $X$, the distinguished triangle $X \to 0 \to X[1] \to$ gives an isomorphism 
\[ f_\RR(X)\otimes f_\RR(X[1])\xrightarrow{\simeq}f_\RR(0)=\RR.\]
The canonical isomorphism
\[
H^\ast(\gamma)_{R\Gamma_{W,c}(U,F)}:f_\ZZ(R\Gamma_{W,c}(U,F))\otimes_\ZZ\RR \xrightarrow{\simeq} f_\RR(R\Gamma_{W,c}(U,F)_\RR)
\]
gives a natural embedding
\[
f_\ZZ(R\Gamma_{W,c}(U,F)) \hookrightarrow f_\RR(R\Gamma_{W,c}(U,F)_\RR)
\]
of the underlying (ungraded) lines, i.e. of abelian groups. 

\begin{rmk}
Define the Picard groupoid of embedded graded lines $\mathcal{P}_{\ZZ\to \RR}$, whose objects are pairs $(f:L\to V,n)$ where $L$ is a free $\ZZ$-modules of rank $1$, $V$ an $\RR$-vector space of dimension $1$ and $n\in \ZZ$, with a map $f:L\to V$ of abelian groups such that the induced map $L\otimes \RR \to V$ is an isomorphism. The projection $\mathcal{P}_{\ZZ\to \RR}\to \mathcal{P}_\ZZ$ and the functor $\mathcal{P}_\ZZ \to \mathcal{P}_{\ZZ\to \RR},~(L,n) \mapsto (L\to L\otimes \RR,n)$ are inverse equivalences of Picard groupoids. The above embedding can be seen as an object in $\mathcal{P}_{\ZZ\to \RR}$.
\end{rmk}

\begin{defi}
Let $F$ be a big or tiny sheaf. The (Weil-étale) Euler characteristic of $F$
is the positive integer $\chi_U(F)$ such that \[\lambda(f_\ZZ(R\Gamma_{W,c}(U,F)))=(\chi_U(F))^{-1}\ZZ\hookrightarrow \RR.\]
\end{defi}

\begin{defi}
Let $F$ be a big or tiny sheaf. We define the rank of $F$:
\[
E_U(F)=\sum_i (-1)^i\cdot i\cdot \dim_\RR(H^i_{W,c}(U,F)_\RR)
\]
\end{defi}

\begin{rmk}
By theorem \ref{regulator_perfect} and the computations of subsection \ref{computations} we have
\begin{align*}
    E_U(F)&=-1(\mathrm{rank}_{\ZZ}H^1_c(U,F)+\mathrm{rank}_{\ZZ}\Ext^1_U(F,\GG_m))+2(\mathrm{rank}_{\ZZ}\Hom_U(F,\GG_m))\\
    &= \mathrm{rank}_\ZZ(\Hom_U(F,\GG_m))-\mathrm{rank}_\ZZ(\Ext^1_U(F,\GG_m))\\
    &= \mathrm{rank}_\ZZ(H^1_c(U,F))-\mathrm{rank}_\ZZ(H^0_c(U,F))
\end{align*}
\end{rmk}

\begin{thm}
The rank and Euler characteristic are respectively additive and multiplicative with respect to big-to-tiny short exact sequences. That is, given a big-to-tiny short exact sequence
\[
0\to F \to G \to H \to 0
\]
we have
\begin{align*}
    E_U(G) &= E_U(F) + E_U(H)\\
    \chi_U(G) &= \chi_U(F)\chi_U(H)
\end{align*}
\end{thm}

\begin{proof}
The result for $E_U$ is immediate given the remark above, the long exact sequence of $\Ext$ groups and the fact that $\Ext^i(F,\GG_m)$ is zero for $i<0$ and torsion for $i>1$.

Let $\Delta$ denote the (non-distinguished) triangle $R\Gamma_{W,c}(U,F)\to R\Gamma_{W,c}(U,G) \to R\Gamma_{W,c}(U,H) \to$. By Theorem \ref{long_exact_cohomolgy_sequence}, $H(\Delta)$ is an exact sequence in $\mathrm{Gr}^b(\mathrm{Mod}_\ZZ^{ft})$, and we can consider the following diagram:
\begin{equation}\label{multiplicativity_euler_characteristic}
\begin{tikzcd}
f_\ZZ(R\Gamma_{W,c}(U,F)) \otimes f_\ZZ(R\Gamma_{W,c}(U,H)) \rar["g_\ZZ(H\Delta)"] \dar & f_\ZZ(R\Gamma_{W,c}(U,G)) \dar \\
(f_\ZZ(R\Gamma_{W,c}(U,F)))_\RR \otimes (f_\ZZ(R\Gamma_{W,c}(U,H))_\RR \rar["(g_\ZZ(H\Delta))_\RR"] \dar["H^\ast(\gamma)_{R\Gamma_{W,c}(U,F)}\otimes H^\ast(\gamma)_{R\Gamma_{W,c}(U,H)}"] & (f_\ZZ(R\Gamma_{W,c}(U,G))_\RR \dar["H^\ast(\gamma)_{R\Gamma_{W,c}(U,G)}"]\\
f_\RR(R\Gamma_{W,c}(U,F)_\RR) \otimes_\RR f_\RR(R\Gamma_{W,c}(U,H)_\RR) \rar["g_\RR(H(\Delta_\RR))=f_\RR(\Delta_\RR)"] \arrow{d}{f_\RR(\tau_F)\otimes f_\RR(\tau_H)} & f_\RR(R\Gamma_{W,c}(U,F)_\RR) \arrow{d}{f_\RR(\tau_G)}\\
f_\RR(D_{F,\RR}[1]\oplus D_{F,\RR}[2]) \otimes_\RR f_\RR(D_{H,\RR}[1]\oplus D_{H,\RR}[2]) \rar & f_\RR(D_{G,\RR}[1]\oplus D_{G,\RR}[2])\\
f_\RR(D_{F,\RR}[1])\otimes_\RR f_\RR(D_{F,\RR}[2]) \otimes_\RR f_\RR(D_{H,\RR}[1]) \otimes_\RR f_\RR(D_{H,\RR}[2]) \rar \uar \dar & f_\RR(D_{G,\RR}[1]) \otimes f_\RR(D_{G,\RR}[2]) \dar\uar\\
\RR \otimes_\RR \RR \rar["mult"] & \RR
\end{tikzcd}
\end{equation}

\begin{itemize}
\item Recall that we have set $f_\ZZ:=H^\ast(g_\ZZ)$, hence the top arrow would be $f_\ZZ(\Delta)$ if $\Delta$ was actually distinguished. Note that $\Delta\otimes \RR$ is distinguished. The top square is a well-defined map in $\mathcal{P}_{\ZZ\to \RR}$, in particular it is a commutative square of abelian groups.

\item Denote $H^\ast_{W,c}(U,-):=H^\ast(R\Gamma_{W,c}(U,-))$. Since $\gamma$ is a morphism of determinants, there is a commutative diagram associated to the exact sequence $H\Delta$:
\[
\begin{tikzcd}
B\circ g_\ZZ(H^\ast_{W,c}(U,F)) \otimes B\circ g_\ZZ(H^\ast_{W,c}(U,H)) \rar["B\circ g_\ZZ(H\Delta)"] \dar["\gamma_{H^\ast_{W,c}(U,F)}\otimes \gamma_{H^\ast_{W,c}(U,H)}"] & B\circ g_\ZZ(H^\ast_{W,c}(U,F)) \dar["\gamma_{H^\ast_{W,c}(U,G)}"]\\
g_\RR(H^\ast_{W,c}(U,F)_\RR) \otimes_\RR g_\RR(H^\ast_{W,c}(U,H)_\RR) \rar["g_\RR((H\Delta)_\RR)"] & g_\RR(H^\ast_{W,c}(U,F)_\RR)
\end{tikzcd}
\]
Unwinding the definitions, we see that the above diagram is exactly the second square.
\item The third square commutes because $\tau$ defines an isomorphism of distinguished triangles.
\item  The commutativity of the fourth square is a formal consequence of the associativity and commutativity axioms of the determinant, as we prove in the next lemma \ref{assocativity_determinant}.
\item The last square commutes by \cite[Lemma 3.6(ii) and its proof]{Breuning11}, and the fact that the unit structure on $\RR=f_\RR(0)$ coming from the distinguished triangle $0\to 0 \to 0 \to$ is given by the multiplication map. Since the image by the multiplication map of $x\ZZ \otimes y \ZZ\subset \RR\otimes \RR$ is $xy\ZZ \subset \RR$, we are done.
\end{itemize}
\end{proof}

\begin{lem}\label{assocativity_determinant}
	Let $[-]:\mathcal{T}\to\mathcal{P}$ be a determinant functor on a triangulated category. Let $\Delta : A\to B \to C \to$ and $\Delta' : A'\to B' \to C'\to $ be distinguished triangles.
	The following diagram, coming from the $3\times 3$ diagram with columns $\Delta, \Delta\oplus \Delta'$, $\Delta'$, commutes:
	\[
	\begin{tikzcd}
		\left[B\oplus B'\right] \rar \dar & \left[B\right]\left[B'\right] \arrow{dd}\\
		\left[A\oplus A'\right]\left[C\oplus C'\right] \dar & \\
		\left[A\right]\left[A'\right]\left[C\right]\left[C'\right] \rar["\text{flip}"] & \left[A\right]\left[C\right]\left[A'\right]\left[C'\right]
	\end{tikzcd}\]
\end{lem}
We will follow below the convention that unmarked arrows will be obtained by functoriality of the determinant, applied to a distinguished triangle, and will go from the middle term to the product of the outer terms. We denote the product by concatenation for compactness.

\begin{proof}
	Consider the following diagram:
	\[
	\begin{tikzcd}
		& {[B\oplus B']} \arrow[ld] \arrow[rd] \arrow[rr] \arrow[dd, "1", phantom] \arrow[rrrd, "2", phantom] &                                                                              & {[B][B']} \arrow[rd]                                                &                                                   \\
		{[A\oplus A'][C\oplus C']} \arrow[rd] \arrow[dd] \arrow[rdd, "3", phantom, shift right=3] &                                                                                                     & {[B\oplus A'][C']} \arrow[ld] \arrow[rr] \arrow[rd] \arrow[dd, "4", phantom] &                                                                     & {[B][A'][C']} \arrow[dd] \arrow[ld, "5", phantom] \\
		& {[A\oplus A'][C][C']} \arrow[rd]                                                                    &                                                                              & {[A][A'\oplus C][C']} \arrow[rd] \arrow[ld] \arrow[d, "6", phantom] &                                                   \\
		{[A][A'][C\oplus C']} \arrow[rr]                                                          & {}                                                                                                  & {[A][A'][C][C']} \arrow[rr,swap, "flip"]                                          & {}                                                                  & {[A][C][A'][C']}                                 
	\end{tikzcd}\]
	The outer diagram is exactly the one we want to prove to be commutative. Square 1 commutes by associativity, applied to the octahedron diagram
	\[
		\begin{tikzcd}
		A\oplus A' \arrow[d, equal] \arrow[r] & B\oplus A' \arrow[d] \arrow[r]    & C \arrow[d]          \\
		A\oplus A' \arrow[r]                                & B\oplus B' \arrow[r] \arrow[d]    & C\oplus C' \arrow[d] \\
		& C' \arrow[r, equal] & C'                  
		\end{tikzcd}
	\]
	Square 2 similarly commutes by the octahedron diagram
	\[
		\begin{tikzcd}
		B \arrow[d, equal] \arrow[r] & B\oplus A' \arrow[d] \arrow[r]    & A' \arrow[d]          \\
		B \arrow[r]                                & B\oplus B' \arrow[r] \arrow[d]    & B' \arrow[d] \\
		& C' \arrow[r, equal] & C'                  
		\end{tikzcd}
	\]
	Square 3 is trivially commutative. Square 4 is obtained by tensoring on the right with $[C']$ the associativity diagram coming from the following octahedron:
	\[
		\begin{tikzcd}
		A \arrow[d, equal] \arrow[r] & A\oplus A' \arrow[d] \arrow[r]    & A' \arrow[d]          \\
		A \arrow[r]                                & B\oplus A' \arrow[r] \arrow[d]    & A'\oplus C \arrow[d] \\
		& C \arrow[r, equal] & C                  
		\end{tikzcd}
	\]
	Square 5 is obtained by tensoring on the right with $[C']$ the associativity diagram coming from the following octahedron:
	\[
	\begin{tikzcd}
	A \arrow[d, equal] \arrow[r] & B \arrow[d] \arrow[r]    & C \arrow[d]          \\
	A \arrow[r]                                & B\oplus A' \arrow[r] \arrow[d]    & A'\oplus C \arrow[d] \\
	& A' \arrow[r, equal] & A'                  
	\end{tikzcd}
	\]
	Finally, square 6 is commutative by the commutativity axiom.
\end{proof}

\subsection{Computations}

We now use the results from section \ref{computations} and appendix A to compute the Weil-étale characteristic.

\begin{defi}[The regulator]\label{regulator}
	Let $F$ bas a $\ZZ$-constructible sheaf. Fix bases modulo torsion of $H^0_c(U,F)$, $H^1_c(U,F)$, $\Ext^1_U(F,\GG_m)$ and $\Hom_U(F,\GG_m)$. Let $R_0(F)$ be the absolute value of the determinant of the pairing
	\[
	H^0_c(U,F)_\RR \times \Ext^1_U(F,\GG_m)_\RR \to \RR
	\]
	and $R_1(F)$ the absolute value of the determinant of the pairing
	\[
	H^1_c(U,F)_\RR \times \Hom_U(F,\GG_m)_\RR \to \RR
	\]
	in those bases. Those quantities do not depend on the choices, and we define the regulator $R(F)$ of $F$:
	\[
	R(F):=\frac{R_1(F)}{R_0(F)}
	\]
\end{defi}

\begin{prop}\label{euler_vs_big}
Let $F$ be a big sheaf. Then
\[
    \chi_U(F)=\frac{R(F)[\Ext^1_U(F,\GG_m)][H^0_c(U,F)]}{[\Hom_U(F,\GG_m)_{tor}][H^1_c(U,F)_{tor}]}
\]
\end{prop}

\begin{prop}\label{euler_vs_tiny}
Let $F$ be a tiny sheaf. Then
\[
\chi_U(F)=\frac{R(F)[\Ext^1_U(F,\GG_m)_{tor}][H^0_c(U,F)_{tor}]}{[\Hom_U(F,\GG_m)][H^1_c(U,F)]} 
\]
\end{prop}

\begin{cor}\label{euler_vs_point}
Let $i:v\hookrightarrow U$ be the inclusion of a closed point of $U$ and $M$ be a finitely generated discrete $G_v$-module. Then $i_\ast M$ is tiny and
    \[
    \chi_U(i_\ast M)= \frac{[H^0(G_v,M)_{tor}]}{[H^1(G_v,M)]R(M)(\log N(v))^{\mathrm{rank}_\ZZ(H^0(G_v,M))}}
    \]
    where $N(v)=|\kappa(v)|$ and $R(M)$ is the determinant of the pairing
    \[
    M^{G_v}_\QQ\times \Hom_{G_v}(M,\QQ)\to \QQ
    \]
    after a choice of bases modulo torsion of $M^{G_v}$ and $\Hom_{G_v}(M,\ZZ)$.
\end{cor}

\begin{rmk}
\begin{itemize}
	\item[]
	\item Note the factor $(\log N(v))^{\mathrm{rank}_\ZZ(H^0(G_v,M))}$ which comes from diagram \eqref{regulateur_point}; there is no sign here because we take absolute values.
	\item The value $\chi_v(M):=\chi_U(i_\ast M)$ doesn't depend on the embedding $i:v\to U$ hence $\chi_v$ is intrinsic to the scheme $v$.
\end{itemize}
\end{rmk}

\begin{prop}\label{euler_vs_zeta}
Suppose $U=\Spec{\mathcal{O}_{K,S}}$ and denote $R_S$, $h_S$ and $\omega$ the $S$-regulator, $S$-class number and number of roots of unity of $K$ respectively. We have
\begin{align*}
    \chi_U(\ZZ)&=\frac{h_SR_S}{\omega} \\
    E_U(\ZZ) &= \mathrm{rank_\ZZ}{\mathcal{O}_{K,S}^\times}=[S]-1
\end{align*}
We obtain from the analytic class number formula for the $S$-zeta function of $K$ (the arithmetic zeta function of $U$) that $E_U(\ZZ)=\mathrm{ord}_{s=0}\zeta_{K,S}$ and the formula for the special value $\lim_{s\to 0} s^{-E_U(\ZZ)}\zeta_{K,S}(s)=-\chi_U(\ZZ)$.\footnote{see for instance \cite[I.2.2]{Tate84}}
\end{prop}

\begin{proof}
In the construction of the regulator pairing, we noticed that for $F=\ZZ$ it simply identifies with the classical regulator ; hence $R(\ZZ)$ is the $S$-regulator $R_S$ of $\mathcal{O}_{K,S}$. Moreover, we have
\begin{align*}
    \Hom_U(\ZZ,\GG_m)_{tor}=(\mathcal{O}_{K,S}^\times)_{tor}=\mu(K)\\
    \Ext^1_U(\ZZ,\GG_m)=H^1_{et}(U,\GG_m)=\mathrm{Pic}(U)
\end{align*}
Finally, the long exact sequence \eqref{long_sequence_compact_support} gives
\[
0 \to H^0_c(U,\ZZ) \to \ZZ \xrightarrow{\Delta} \prod_{v\in S}\ZZ \to H^1_c(U,\ZZ) \to H^1_{et}(U,\ZZ)
\]
where $\Delta$ is the diagonal inclusion. We already saw that $H^1_{et}(U,\ZZ)=0$. Thus $H^1_c(U,\ZZ)\simeq \prod_{v\in S}\ZZ/\ZZ$ is free and $H^0_c(U,\ZZ)=0$. We obtain
\begin{align*}
E_U(F)&=\mathrm{rank}_\ZZ(\Hom_U(\ZZ,\GG_m))-\mathrm{rank}_\ZZ(Ext^1_U(\ZZ,\GG_m))=\mathrm{rank}_\ZZ(\mathcal{O}_{K,S}^\times)\\
\chi_U(F)&=\frac{R_S[\mathrm{Pic}(U)]}{[\mu(K)]}=\frac{h_SR_S}{\omega}.
\end{align*}

\end{proof}

\begin{prop}\label{euler_vs_zeta_proper_curve}
Let $U=C$ be the smooth proper curve associated to the function field $K$, and let $\FF_q$ be its field of constants. We have
\begin{align*}
\chi_C(\ZZ) &= \frac{h}{\omega\log(q)}\\
E_C(\ZZ) &= -\mathrm{rank}_\ZZ \mathrm{Pic}(C) = -1
\end{align*}
where $h=[\mathrm{Pic}^0(C)]$ is the cardinality of the group of classes of degree $0$ divisors and $\omega$ the number of roots of unity of $K$.\footnote{Note that $\omega=q-1$} We verify again that $\mathrm{ord}_{s=0}\zeta_{C}=E_C(\ZZ)$ and $\zeta_C^\ast(0)=-\chi_C(\ZZ)$ for $\zeta_C$ the arithmetic zeta function of $C$.\footnote{see for instance \cite[VII.6, Proposition 4]{Weil}}
\end{prop}

\begin{proof}
The constant sheaf $\ZZ$ is tiny on $C$ because $H^1(C,\ZZ)=0$, so we use proposition \ref{euler_vs_tiny}. We consider $C$ as a curve over the field of constants $\FF_q$ of its function field, hence the degree map is surjective \cite[VII.5, cor. 5]{Weil} and we have an exact sequence 
		\[0 \to \mathrm{Pic}^0(C) \to \mathrm{Pic}(C) \xrightarrow{deg} \ZZ \to 0\]
	with $\mathrm{Pic}^0(C)$ the class group of degree zero divisors, which is finite by \cite[IV.4 Thm 7]{Weil}. We have $H^0(C,\ZZ)=\ZZ$ and $H^0(C,\GG_m)=\FF_q^\times$.  We obtain
	\[
		\chi_U(\ZZ)=\frac{R(\ZZ)[Pic^0(C)]}{q-1}
	\]
	Finally, by the discussion in paragraph \ref{case_Z}, the rational regulator pairing for $C$ identifies via the degree map to the trivial pairing $\QQ 			\times \QQ \to \QQ$ with the natural integral bases, hence the comparison between the regulator pairing and the rational regulator pairing gives $R(\ZZ)=1/\mathrm{log}(q)$.
\end{proof}

\subsection{Functoriality of the Euler characteristic}\label{functoriality_euler_characteristic}

\begin{prop}
Let $L/K$ be a finite Galois extension and $\pi:V \to U$ the normalization of $U$ in $L$. If $F$ is a big or tiny sheaf on $V$, so is $\pi_\ast F$ and
\begin{align*}
    \chi_U(\pi_\ast F) &= \chi_V(F)\\
    E_U(\pi_\ast F) &= E_V(F)
\end{align*}
\end{prop}

\begin{proof}
As seen in paragraph \ref{cas_loc_const}, we have isomorphisms
\begin{align*}
    R\Hom_U(\pi_\ast F,\GG_m) &\simeq R\Hom_V(F,\GG_m)\\
    R\Gamma_c(U,\pi_\ast F) & \simeq R\Gamma_c(V,F)
\end{align*}
so the first assertion is immediate, as well as the equality $E_U(\pi_\ast F) = E_V(F)$. Similarly, there is also an isomorphism 
\[R\hat{\Gamma}_c(U,\pi_\ast F) \simeq R\hat{\Gamma}_c(V,F)\]
Write
\begin{align*}
    & D^U_F:=R\Hom(R\Hom_U(F,\GG_m),\QQ[-3])\\
    & E^U_F:=R\Hom(R\Hom_U(F,\GG_m),\QQ/\ZZ[-3])\\
    & \beta^U_F ~~\text{for the arrow constructed in section \ref{construction_weil_etale}}
\end{align*}
We have seen in paragraph \ref{cas_loc_const} that the above isomorphisms induce isomorphisms of the regulator pairing
\[
\begin{tikzcd}
R\Gamma_c(V, F)\otimes R\Hom_V(F,\GG_m) \dar["\simeq"] \rar & \RR[-1] \dar[equal]\\
R\Gamma_c(U,\pi_\ast F)\otimes R\Hom_U(\pi_\ast F,\GG_m) \rar & \RR[-1]
\end{tikzcd}
\]
and Artin-Verdier pairing
\[
\begin{tikzcd}
R\hat{\Gamma}_c(V, F)\otimes R\Hom_V(F,\GG_m) \dar["\simeq"] \rar & \QQ/\ZZ[-3] \dar[equal]\\
R\hat{\Gamma}_c(U,\pi_\ast F)\otimes R\Hom_U(\pi_\ast F,\GG_m) \rar & \QQ/\ZZ[-3]
\end{tikzcd}
\]
Thus there is a commutative diagram
\[
\begin{tikzcd}
D^V_{F} \rar \dar["\simeq"] & E^V_{ F} \dar["\simeq"] & R\hat{\Gamma}_c(V, F) \dar["\simeq"] \lar & R\Gamma_c(V, F) \lar \dar["\simeq"]\\
D^U_{\pi_\ast F} \rar & E^U_{\pi_\ast F} & R\hat{\Gamma}_c(U,\pi_\ast F) \lar & R\Gamma_c(U,\pi_\ast F) \lar 
\end{tikzcd}
\]
Since $\beta^U_{\pi_\ast F}$ and $\beta^V_F$ are defined uniquely by their cohomology in degree $3$ (and more generally morphisms from $D^V_F$ to $R\Gamma_c(U,\pi_\ast F)$ verify this property, by the same argument as for $\beta$), it follows that the following diagram is commutative
\[
\begin{tikzcd}
D^V_F \rar["\beta^V_F"] \dar["\simeq"] & R\Gamma_c(V,F) \dar["\simeq"]\\
D^U_{\pi_\ast F} \rar["\beta^U_{\pi_\ast F}"] & R\Gamma_c(V,F)
\end{tikzcd}
\]
Thus there exists an induced isomorphism on the cones, giving an isomorphism of distinguished triangles:
\[
\begin{tikzcd}
D^V_F \rar \dar["\simeq"] & R\Gamma_c(V,F) \dar["\simeq"] \rar & R\Gamma_{W,c}(V,F)\dar[dotted,"\simeq"] \\
D^U_{\pi_\ast F} \rar["\beta^U_{\pi_\ast F}"] & R\Gamma_c(V,F) \rar & R\Gamma_{W,c}(U,\pi_\ast F)
\end{tikzcd}
\]
The induced isomorphism is moreover unique, similarly to how we argued for the uniqueness of the complex $R\Gamma_{W,c}$. All the identifications we have made are compatible and we obtain $\chi_U(\pi_\ast F)=\chi_V(F)$ by taking determinants. 
\end{proof}

\begin{prop}\label{charac_extension_zero}
Let $j: V\subset U$ be an open immersion. If $F$ is a big or tiny sheaf on $V$, then so is $j_!F$, and we have
\begin{align*}
    \chi_U(j_! F) &= \chi_V(F)\\
    E_U(j_! F) &= E_V(F) \\
\end{align*}
\end{prop}

\begin{proof}
The proof proceeds as the previous one, using Proposition \ref{milne_II.2.3}, the adjunction $j_!\dashv j^\ast$ and the results from paragraph \ref{case_extension_zero}.
\end{proof}

\subsection{The \texorpdfstring{$L$}{L}-function of a \texorpdfstring{$\ZZ$}{Z}-constructible sheaf}

Let $X$ be an irreducible scheme of finite type over $\ZZ$ and of dimension $1$, and let $F$ be a $\ZZ$-constructible sheaf on $X$. We define the $L$-function associated to $F$ and go over its functorial properties.

\begin{defi}
Let $i:v\to X$ be a closed point of $X$. The local factor with respect to $F$ and $v$ is
\[L_v(F,s)=\det(I-N(v)^{-s}\mathrm{Frob}_v|F_{v}\otimes \CC)^{-1}\] where $F_v$ is the discrete $G_v:=\mathrm{Gal}(\kappa(v)^{sep}/\kappa(v))$-module corresponding to $i^\ast F$ and $\mathrm{Frob}_v \in G_v$ is the (arithmetic) Frobenius.

The L-function associated to $F$ is
\[L_X(F,s)=\prod_{v \in X_{0}}L_v(F,s)\]
where $X_0$ is the set of closed points of $X$. It is well-defined and holomorphic for $\Re(s)>1$.
\end{defi}

Suppose that $X$ is regular, let $K$ be the function field of $X$ and $G_K$ its absolute Galois group. Our definition differs only at a finite number of local factors from the Artin L-function of the complex $G_K$-representation given by $F_\eta\otimes \CC$, where $F_\eta$ is the pullback of $F$ to $\Spec(K)$. Both definitions coincide when $F=\ZZ$, and in particular they equal the (arithmetic) zeta function of $X$:
\[
L_X(\ZZ)=\zeta_{K,S}
\]

On the other hand if $i:v\to X$ is the inclusion of a closed point and $M$ is a discrete $G_v$-module of finite type, we have
\[
L_X(i_\ast M,s)=\det(I-N(v)^{-s}\mathrm{Frob}_v|M\otimes \CC)^{-1}
\]

Let us investigate the functorial properties of the L-functions of $\ZZ$-constructible sheaves.
\begin{prop}
\begin{enumerate}
    \item If $0\to F \to G \to H \to 0$ is an exact sequence of $\ZZ$-constructible sheaves, then \[L_X(G)=L_X(F)L_X(H).\]
    \item \label{L-function_finite_morphism} Let $\pi:Y\to X$ be a finite morphism, with $Y$ irreducible. If $F$ is a $\ZZ$-constructible sheaf on $Y$, then $\pi_\ast F$ is $\ZZ$-constructible and
    \begin{equation}
        L_X(\pi_\ast F)=L_Y(F).
    \end{equation}
    \item Let $j:V\subset X$ an open subscheme and let $F$ be a $\ZZ$-constructible sheaf on $V$. Then
    \[
    L_X(j_! F)=L_V(F).
    \]
\end{enumerate}
\end{prop}
\begin{proof}
Only \ref{L-function_finite_morphism}. is non-trivial. The pushforward $\pi_\ast$ commutes with base change so we have
\[
(\pi_\ast F)_v=\prod_{w|v} \mathrm{ind}^{G_w}_{G_v}(F_w)
\]
where the product is on finite places $w$ above $v$, hence
\[
L_v(\pi_\ast F,s)=\prod_{w|v} L_w(F,s)
\]
since local factors behave well with respect to induced modules \cite[prop. VII.10.4.(iv)]{ANT}.
\end{proof}

\begin{cor}
The L-function of a $\ZZ$-constructible sheaf on $X$ is meromorphic on $\CC$.
\end{cor}
\begin{proof}
By the above functoriality properties and the dévissage argument, we reduce to checking the corollary for zeta functions of regular $X$, for local factors, and for the $L$-functions of constructible sheaves, which are constant equal to $1$ hence trivially meromorphic.
\end{proof}

\begin{defi}
	Let $F$ be a $\ZZ$-constructible sheaf on $X$.
	We let 
	\[
	r_X(F):=\mathrm{ord}_{s=0}L_X(F,s)
	\]
	the order of $L_X(F,s)$ at $s=0$ and 
	\[L^\ast_X(F,0)=\lim_{s\to 0}L_X(F,s)s^{-r_X(F)}
	\]
	the special value of $L_X(F,s)$ at $s=0$.
\end{defi}
We will relate in sections \ref{section_special_value_thm} and \ref{twisted_L_functions} the order and special value at $s=0$ to the previously constructed Euler characteristic, by reducing via the dévissage argument to special cases.
\subsection{Extending the definition of the Weil-étale Euler characteristic}
In this section we construct a well-defined functorial Euler characteristic for any $\ZZ$-constructible sheaf, bypassing a full construction of the Weil-étale complex, which will enable us to obtain a special values theorem.

\begin{defi}\label{def_euler_char}
Let $F$ be a $\ZZ$-constructible sheaf on $U$. Let $j:V\subset U$ be an open affine non-empty subscheme such that $F_{|V}$ is locally constant and let $i:Z:=U \backslash V \hookrightarrow U$ be the inclusion of the closed complement. Denote $F_V:=j_!(F_{|V})$ and $F_Z:=i_\ast i^\ast F$. By the remark below Definition \ref{def_big_tiny} and by Proposition \ref{charac_extension_zero}, $F_V$ is big and $F_Z$ is tiny. We thus define the rank of $F$ and Weil-étale Euler characteristic of $F$:
\begin{align*}
    E_U(F) &= E_U(F_V)+E_U(F_Z)\\
    \chi_U(F) &= \chi_U(F_V)\chi_U(F_Z)
\end{align*}
\end{defi}

\begin{rmk}
	If $F$ is big or tiny, the above exact sequence is big-to-tiny, hence this definition is coherent with the previous one for $F$.
\end{rmk}

We first show that those quantities do not depend on the choice of $V$:

\begin{proof}
Let $j':V'\subset U$ be another open affine subscheme with $F_{V'}$ locally constant and let $i':Z'\hookrightarrow U$ be its closed complement. Without loss of generality we can suppose $V'\subset V$. Let $t:T:=V\backslash V' \hookrightarrow U$ be the closed inclusion. The open-closed decomposition lemma gives big-to-tiny short exact sequences
\[
\begin{tikzcd}[sep=small]
0 \rar & (t_\ast t^\ast F) \rar & i'_\ast i'^\ast F \rar \rar & i_\ast i^\ast F \rar & 0\\
0 \rar & j'_!F_{|V'} \rar & j_!F_{|V} \rar & t_\ast t^\ast F \rar & 0
\end{tikzcd}
\]
The proof follows formally.
\end{proof}

We can now investigate functoriality properties:

\begin{prop}
Let $0\to F\to G \to H \to 0$ be a short exact sequence of $\ZZ$-constructible sheaves. Then
\[
\chi_U(G)=\chi_U(F)\chi_U(H)
\]
\end{prop}

\begin{proof}
Choose a common open subscheme $V$ on which the restriction of each sheaf is locally constant, then apply the open-closed decomposition lemma, giving a commutative diagram with exact rows and columns (using notation from Definition \ref{def_euler_char}):
\[
\begin{tikzcd}
 & 0 \dar & 0 \dar & 0 \dar & \\
0 \rar & F_V \rar \dar & F \rar \dar & i_\ast F_Z \rar \dar & 0\\
0 \rar & G_V \rar \dar & G \rar \dar & i_\ast G_Z \rar \dar & 0\\
0 \rar & H_V \rar \dar & H \rar \dar & i_\ast H_Z \rar \dar & 0\\
 & 0 & 0 & 0 \\
\end{tikzcd}
\]
The terms on the left are big sheaves and the terms on the right are tiny sheaves, so the proof follows formally.
\end{proof}

\begin{prop}
Let $L/K$ be a finite Galois extension and $\pi:V \to U$ the normalization of $U$ in $L$. If $F$ is a $\ZZ$-constructible sheaf on $V$, then
\begin{align*}
    \chi_U(\pi_\ast F) &= \chi_V(F)\\
    E_U(\pi_\ast F) &= E_V(F)
\end{align*}
\end{prop}

\begin{proof}
We can fix open affine subschemes $j:U'\subset U$, $j':V'\to V$ such that $F_{|V'}$ is locally constant and $\pi':=\pi_{|V'}:V'\to U'$ is the normalization of $U'$ in $L$. Then $(\pi_\ast F)_{|U'}=\pi'_\ast (F_{|V'})$ is locally constant. Denote $i:Z=U\backslash U' \hookrightarrow U$ and $i':Z'=V\backslash V' \hookrightarrow V$ the closed inclusions. The open-closed decomposition lemma gives short exact sequences
\begin{align*}
    & 0 \to (\pi_\ast F)_{U'} \to \pi_\ast F \to (\pi_\ast F)_Z \to 0\\
    & 0 \to F_{V'} \to F \to F_{Z'} \to 0
\end{align*}
Denote by $\pi_Z$ the base change of $\pi$ to $Z$. Then $\pi$ is finite so $i_\ast i^\ast \pi_\ast = i_\ast \pi_{Z,\ast} i'^\ast = \pi_\ast i'_\ast i'^\ast F$ thus $(\pi_\ast F)_Z=\pi_\ast F_{Z'}$ and we get
\[
    \chi_U(\pi_\ast F):=\chi_{U'}(\pi'_\ast F)\chi_U( (\pi_\ast F)_Z)=\chi_V'(F)\chi_U(\pi_\ast F_{Z'})
    =\chi_V'(F)\chi_V(F_{Z'})
    =:\chi_V(F)
\]
\end{proof}

\begin{prop}
Let $j: V\subset U$ be an open immersion. Let $F$ be a $\ZZ$-constructible sheaf on $V$. We have
\begin{align*}
    \chi_U(j_!F) &= \chi_V(F)\\
    E_U(j_!F) &= E_V(F) \\
\end{align*}
\end{prop}

\begin{proof}
Let $j':V'\subset V$ an open affine subscheme with $F_{|V'}$ locally constant, and let $i':Z':=V\backslash V'\hookrightarrow V$, $i:Z=U\backslash V' \hookrightarrow U$ denote the closed inclusions. Then $(j_!F)_{|V'}=F_{|V'}$, $i^\ast j_! F = t_\ast i'^\ast F$ (with $t$ the clopen immersion $Z'\subset Z$) and $i_\ast t_\ast=j_!i'_\ast$ so the short exact sequences given by the open-closed decomposition lemma give
\begin{align*}
    \chi_U(j_!F):=\chi_U((jj')_! (j_!F)_{|V'})\chi_U(i_\ast i^\ast j_! F) =\chi_U((jj')_! F_{|V'})\chi_U(j_!i'_\ast i'^\ast F) &=\chi_V(j'_! F_{|V'})\chi_V(i'_\ast i'^\ast F)\\
    &=:\chi_V(F)
\end{align*}
\end{proof}

\subsection{Explicit formula for the Weil-étale Euler characteristic}

\begin{prop}\label{euler_vs_arbitrary}
Let $F$ be a $\ZZ$-constructible sheaf. Then
\[
\chi_U(F)=\frac{[H^0_c(U,F)_{tor}][\Ext^1_U(F,\GG_m)_{tor}]R(F)}{[H^1_c(U,F)_{tor}][\Hom_U(F,\GG_m)_{tor}]}
\]
\end{prop}

\begin{proof}
Let $F_V$, $F_Z$ be as in the preceding section, so that $\chi_U(F):=\chi_U(F_V)\chi_U(F_Z)$. Consider the two following exact sequences, viewed as acyclic cochain complexes :
\begin{align*}
	A^\bullet : 0 &\to H^0_c(U,F_V) \to H^0_c(U,F) \to H^0_c(U,F_Z)\\
	& \to H^1_c(U,F_V) \to H^1_c(U,F) \to H^1_c(U,F_Z) \to I \to 0\\
	B^\bullet : 0 & \to \Hom_U(F_Z,\GG_m) \to \Hom_U(F,\GG_m) \to \Hom_U(F_V,\GG_m)\\
	 & \to \Ext^1_U(F_Z,\GG_m) \to \Ext^1_U(F,\GG_m) \to \Ext^1_U(F_V,\GG_m) \to J \to 0
\end{align*}
where $I$ is the image of $H^1_c(U,F_Z)$ in $H^2_c(U,F_V)$ and $J$ is the image of $\Ext^1_U(F_V,\GG_m)$ in $\Ext^2_U(F_Z,\GG_m)$. We choose the convention of putting the first non-zero term of $A$ in degree $0$ and the last non-zero term of $B$ in degree $2$. Since $F_V$ is big and $F_Z$ is tiny, using Artin-Verdier duality we have isomorphisms between finite groups $H^1_c(U,F_Z)\simeq \Ext^2_U(F_Z,\GG_m)^D$ and $H^2_c(U,F_V) \simeq \Ext^1_U(F_V,\GG_m)^D$, thus Pontryagin duality gives $I\simeq J^D$.

The regulator pairing gives an isomorphism $\phi : B_\RR \to \Hom(A_\RR,\RR)[-1]$,\footnote{hence the choice of indexing convention} so lemma \ref{formula_determinant_acyclic_complexes} together with propositions \ref{euler_vs_tiny} and \ref{euler_vs_tiny} give:
\begin{align*}
1 &= \frac{[H^0_c(U,F)_{tor}][\Ext^1_U(F,\GG_m)_{tor}]R(F)}{[H^1_c(U,F)_{tor}][\Hom_U(F,\GG_m)_{tor}]}\cdot\frac{[J]}{[I]}\cdot\frac{1}{\chi_U(F_V)\chi_U(F_Z)}\\
& = \frac{[H^0_c(U,F)_{tor}][\Ext^1_U(F,\GG_m)_{tor}]R(F)}{[H^1_c(U,F)_{tor}][\Hom_U(F,\GG_m)_{tor}]}\cdot\frac{1}{\chi_U(F)}
\end{align*}
where $R(F)$ is the regulator of $F$ of definition \ref{regulator}; the result follows.
\end{proof}

\subsection{Computations, part II, and the special value theorem}\label{section_special_value_thm}

Our aim is to prove Theorem \ref{thmC} by using the dévissage argument to reduce to special cases. The case of $F=\ZZ$ was already shown; let us handle now the case of a constructible sheaf.

\begin{prop}\label{euler_constructible}
Let $F$ be a constructible sheaf. Then $\chi_U(F)=1$ and $E_U(F)=0$. As the $L$-function of a constructible sheaf is $1$, we have $r_U(F)=E_U(F)$ and $L^\ast_U(F,0)=\chi_U(F)$.
\end{prop}

\begin{proof}
Let us first treat the case where $F$ is supported on a finite set ; by functoriality, this reduces to the case where $F$ is of the form $i_\ast M$, with $i:v\to U$ the inclusion of a closed point and $M$ a finite discrete $G_v$-module. By Corollary \ref{euler_vs_point}, we have
\[
\chi_U(i_\ast M)=\frac{[H^0(G_v,M)]}{[H^1(G_v,M)]}
\]
Since $M$ is finite, we have \cite[XIII.1 prop. 1]{Serre80} that $H^1(G_v,M)=M/(\varphi-1)M$, where $\varphi$ is the Frobenius. Hence the exact sequence
\begin{align}\label{cohomology_Z_hat_finite_coeff}
0 \to H^0(G_v,M)\to M \xrightarrow{\varphi-1} M \to H^1(G_v,M) \to 0
\end{align}
gives the result by taking cardinals.

Let us now treat the general case. We have seen in paragraph \ref{cas_constructible} that the regulator pairing of $F$ is trivial after base change to $\RR$, thus $R(F)=1$ and $E_U(F)=0$. A constructible sheaf is big, so by proposition \ref{euler_vs_big} and Artin-Verdier duality we have:
\[
\chi_U(F)=\frac{[\hat{H}^2_c(U,F)][H^0_c(U,F)]}{[\hat{H}^3_c(U,F)][H^1_c(U,F)]}
\]

Suppose we are in the number field case. We want to apply \cite[II.2.13]{ADT} ; by functoriality we can throw away a finite number of points, and thus we can assume that there is an integer $m$ invertible on $U$ such that $mF=0$. Thus using \cite[II.2.13]{ADT}, we obtain further
\[
\chi_U(F)=\frac{[\hat{H}^1_c(U,F)][H^0_c(U,F)]}{[\hat{H}^0_c(U,F)][H^1_c(U,F)]}\prod_{v \in S_\infty}[H^0(K_v,F_{\eta_v})]
\]
Recall that the cofiber $T$ of $R\Gamma_c(U,-)\to R\hat{\Gamma}_c(U,-)$ computes homology at the archimedean places (equation \ref{equation_T}). For a cyclic group $G$ with generator $s$ and a $G$-module $M$, the exact sequence
\[
0\to H^0(G,M) \to M \xrightarrow{s-1} M \to H_0(G,M) \to 0
\]
shows that $[H_0(G,M)]=[H^0(G,M)]$, thus
\[
[H^0(T_F)]=\prod_{v \in S_\infty}[H^0(K_v,F_{\eta_v})]
\]
Moreover, the Herbrand quotient of a finite $G$-module is $1$ so we have, using the exact sequence \eqref{long_sequence_compact_support}:
\[
[\hat{H}^{-1}_c(U,F)]=\prod_{v\in S_\infty}[\hat{H}^{-1}(K_v,F_{\eta_v}))]=\prod_{v\in S_\infty}[\hat{H}^{-2}(K_v,F_{\eta_v}))]=[H^{-1}(T_F)]
\]
The proposition now follows from the long exact cohomology sequence
\[
0 \to \hat{H}^{-1}_c(U,F)\to H^{-1}(T_F) \to H^0_c(U,F) \to \hat{H}^0_c(U,F) \to H^0(T_F) \to H^1_c(U,F) \to \hat{H}^1_c(U,F) \to 0
\]
by taking cardinals.

In the function field case, by functoriality we can suppose that $U=C$ is a proper smooth curve over $k=\FF_p$. Denote $\bar{X}:=X\times_k\bar{k}$ where $\bar{k}$ is an algebraic closure of $k$, and $\bar{F}$ the pullback of $F$ to $\bar{X}$. Then the cohomology groups $H^i(\bar{X},\bar{F})$ are finite \cite[VI.2.1]{MilneEtCohom}. For $M$ a finite $G_k:=\mathrm{Gal}(\bar{k}/k)$-module, we have $H^i(G_k,M)=0$ for $i\geq 2$ thus the spectral sequence
	\[
	E_2^{p,q}=H^p(G_k,H^q(\bar{X},\bar{F}))\Rightarrow H^{p+q}(X,F)
	\]
	degenerates at the $E_2$ page to give short exact sequences
	\[
	0 \to H^{1}(G_k,H^{i-1}(\bar{X},\bar{F})) \to H^{i}(X,F) \to H^{i}(\bar{X},\bar{F})^{G_k} \to 0
	\]
	for all $i\in \ZZ$.
	From \eqref{cohomology_Z_hat_finite_coeff} we deduce moreover that $[H^{1}(G_k,H^{i}(\bar{X},\bar{F}))]=[H^{i}(\bar{X},\bar{F})^{G_k}]$, which concludes the proof.
\end{proof}

We now prove \ref{thmC}. For a $\ZZ$-constructible sheaf $F$, denote $F_\eta$ the discrete finite type $G_K:=\mathrm{Gal}(K^{sep}/K)$-module obtained by pullback to the generic point $\eta=\Spec(K)$, and $F(K):=H^0(G_K,F_\eta)$.

\begin{thm}[Special values formula for $\ZZ$-constructible sheaves]\label{special_values_thm}
Let $F$ be a $\ZZ$-constructible sheaf. Then
\begin{align*}
    r_U(F) &= E_U(F)\\
    L^\ast_U(F,0)&=(-1)^{\mathrm{rank}_\ZZ F(K)}\chi_U(F)
\end{align*}
\end{thm}

\begin{proof}
Let us remark that the quantity $\mathrm{rank}_\ZZ F(K)$ behaves well:

\begin{lem}
	\begin{enumerate}
		\item If $0\to F\to G \to H\to 0$ is a short exact sequence of $\ZZ$-constructible sheaves then $\mathrm{rank}_\ZZ G(K)=\mathrm{rank}_\ZZ F(K)+\mathrm{rank}_\ZZ H(K)$.
		\item Let $\pi:U'\to U$ denote the normalization of $U$ in a finite Galois extension $L/K$ and let $F$ be a $\ZZ$-constructible sheaf on $U'$. Then $(\pi_\ast F)(K)=F(L)$.
		\item Let $j:V\subset U$ be an open subscheme and let $F$ be a $\ZZ$-constructible sheaf on $V$. Then $(j_!F)(K)=F(K)$.
	\end{enumerate}
\end{lem}

\begin{proof}
The first point comes from the fact that Galois cohomology is torsion, which gives a short exact sequence
\[
0 \to F(K)\otimes Q \to G(K)\otimes \QQ \to H(K)\otimes \QQ \to 0.
\]
The two others are straightforward computations.
\end{proof}

Given now the functorial properties of the quantities involved, using the dévissage argument we reduce to the special cases of a constructible sheaf, which is the previous proposition, of the constant sheaf $\ZZ$ on $\Spec(\mathcal{O}_{K,S}$ or on a regular proper curve, which are prop. \ref{euler_vs_zeta} and prop. \ref{euler_vs_zeta_proper_curve}, and of a pushforward $i_\ast M$ of a finite type torsion-free discrete $G_v$-module for $i:v\hookrightarrow U$ a closed point.

Let $v=\Spec(\FF_q)$ be the spectrum of a finite field and let $M$ be a finite type $G_v$-module. Define the Weil-étale Euler characteristic of $M$ by $\chi_v(M):=\chi_{\Spec(\ZZ)}(f_\ast M)$ where $f:v \to \Spec(\ZZ)$ is the structural morphism. Then $f$ decomposes as $v \xrightarrow{\pi} v_0 \xrightarrow{i} \Spec(\ZZ)$ with $v_0=\Spec(\FF_p)$. Since $\pi$ is finite étale, we have 
\[
Rf^! \GG_m = \pi^\ast Ri^! \GG_m=\pi^\ast \ZZ[-1]=\ZZ[-1]
\]
Moreover $R\Gamma_c(\Spec(\ZZ),f_\ast M)=R\Gamma(v,M)$ and the canonical map $\ZZ=H^0(G_v,\ZZ) \to H^0(G_{v_0},\ZZ)=\ZZ$ is multiplication by $[\FF_q:\FF_p]$. It follows using prop. \ref{euler_vs_arbitrary} and \ref{euler_vs_point}, and an analysis similar to the one in paragraph \ref{case_sheaf_closed_point} (to compute $R(f_\ast M)$) that $\chi_v(M)=\chi_U(i_\ast M)$ as soon as $i:v\hookrightarrow U$ realizes $v$ as a closed point of $U$.
If $\pi : w=\Spec(\FF_{q^n}) \to v$ is a finite morphism and $M$ is a finite type $G_w$-module, we immediately find $\chi_v(\mathrm{ind}_{G_v}^{G_w}M)=\chi_w(M)$. 

We have to show that $\mathrm{ord}_{s=0}L_v(M,s)=E_U(i_\ast M)$ and
\[
L_v^\ast(M,0)=\chi_v(M)
\]
for $M$ a finite type torsion-free discrete $G_v$-module; since $\chi_v$ and $E_U(i_\ast M)=\mathrm{rank}_\ZZ H^0(G_v,M)$ behave well with induction, and the case of $M$ finite was already treated, by Artin induction we reduce to $M=\ZZ$. Then $L_v(\ZZ,s)=1/(1-q^{-s})$, so $L_v^\ast(\ZZ,0)=1/\log q$ with order $1$; on the other hand, $\mathrm{rank}_\ZZ H^0(G_v,\ZZ)=1$ and prop. \ref{euler_vs_point} gives $\chi_v(\ZZ)=1/\log(q)$. This concludes the proof.
\end{proof}

\begin{rmk}
\begin{itemize}
\item[]
\item The identification of the sign in Theorem \ref{special_values_thm} was inspired by \cite{Geisser20}.
\item The formula for proper smooth curves was already shown in \cite[3.1]{Geisser20}; note that the $\log(q)$ factor in \textit{ibid.} is integrated in our Weil-étale Euler characteristic, coming from the regulator pairing.
\end{itemize}
\end{rmk}

\section{Twisted Artin \texorpdfstring{$L$}{L}-functions of integral \texorpdfstring{$G_K$}{GK}-representations}\label{twisted_L_functions}

We obtain as a corollary a result generalizing and precising \cite{Tran16} ; we allow singular points, we work also on curves, there is no need for a correcting factor for the $2$-torsion, and we identify the sign.

Let $K$ be a global field and $X$ be an irreducible scheme, finite type over $\ZZ$, of dimension $1$ and with residue field at the generic point $K$. The morphism $X\to \Spec(\ZZ)$ is either dominant or factors through $\Spec(\FF_p)$, so $X$ is either a curve over $\FF_p$ or $X\to \Spec(\ZZ)$ is quasi-finite, hence by Zariski's main theorem is the open subscheme of the spectrum of a finite $\ZZ$-algebra, thus if $X$ is reduced of an order in $K$. The singular locus of $X$ will be denoted $Z$ and the regular locus $U$.

Following \cite{Deninger87}, define the complex of sheaves
\[
\mathbb{G}_X=\left[ (g_\eta)_\ast \GG_{m,\eta} \to \bigoplus_{v\in X_{0}}(i_v)_\ast \ZZ \right]
\]
where $X_{0}$ denotes the closed points, $\eta$ is the generic point and $g_\eta$ (resp. $i_v$) is the canonical morphism $\Spec(\mathcal{O}_{X,\eta}) \to X$ (resp. $v\to X$). If $X$ is regular, we have $\mathbb{G}_X\simeq \GG_{m,X}$ \cite[21.6.9]{EGAIV4}.

Let $\pi: Y\to X_{red}\to X$ denote the normalization of $X_{red}$. Define cohomology with compact support as in \cite[II.6]{ADT}, through the fiber sequence:
\[
R\Gamma_c(X,-) \to R\Gamma(X,-) \to \prod_{v \in S} R\Gamma(K_v,(-)_{\eta_v})
\]
where $S$ is the set of places of $K$ not corresponding to a point of $Y$. The corresponding cohomology groups will be denoted again $H^i_c(X,-)$; propositions \ref{milne_II.2.3} and \ref{compact_support_finite_type}, as well as Artin-Verdier duality (theorem \ref{AVduality}) still hold ; this is easily seen by comparing $X$ to its regular locus $U$.

\begin{cor}\label{special_value_zeta_singular}
	Let $f : X \to T$ denote either a quasi-finite morphism $X\to T:=\mathbb{P}^1_{\FF_p}$ (in the function field case), which exists by the projective Noether normalization lemma, or the quasi-finite structural morphism $X\to T:=\Spec(\ZZ)$ (in the number field case); then $f_!$ is well-defined. Let $F$ be a $\ZZ$-constructible sheaf on $X$. We have the order and special value formula at $s=0$:
	\begin{align*}
	& r_{X}(F)=E_T(f_! F)\\
	& L_{X}^\ast(F,0)=(-1)^{\mathrm{rank}_\ZZ F(K)} \chi_T(f_! F)
	\end{align*}
	Explicitly, we have:
	\begin{align*}
	E_T(f_!F) &= \mathrm{rank}_\ZZ H^1_c(X,F)-\mathrm{rank}_\ZZ H^0_c(X,F) = \mathrm{rank}_\ZZ \Hom_X(F,\GG_X)-\mathrm{rank}_\ZZ \Ext^1(F,\GG_X)\\
	\chi_T(f_!F)&= \frac{R(F)[\Hom_X(F,\GG_X)_{tor}][H^0_c(X,F)_{tor}]}{[\Ext^1_X(F,\GG_X)_{tor}][H^1_c(X,F)_{tor}]}
	\end{align*}
	where $R(F):=\frac{R_1(F)}{R_0(F)}$ with $R_0(F)$ is the absolute value of the determinant of the pairing
	\[
	H^0_c(X,F)_\RR \times \Ext^1_U(F,\GG_X)_\RR \to \RR
	\]
	and $R_1(g_\ast M)$ the absolute value of the determinant of the pairing
	\[
	H^1_c(X,F)_\RR \times \Hom_X(F^,\GG_X)_\RR \to \RR
	\]
	in bases modulo torsion.
\end{cor}

\begin{proof}
	We have $L_T(f_!F,s)=L_X(F,s)$ and $R\Gamma_c(T,f_ !F)=R\Gamma_c(X,F)$. The finite base change theorem implies that 
	\[
	Rf^! \GG_{m,T}=Rf^! \GG_T = \GG_X,
	\]
	which shows that $R\Hom_T(f_!g_\ast M,\GG_m)=R\Hom_X(g_\ast M, \GG_X)$. Theorem \ref{special_values_thm} then gives the first result, and for the second it suffices to apply proposition \ref{euler_vs_arbitrary}.
\end{proof}

Denote $g:\Spec(K) \to X$ the canonical morphism, and let $M$ be a finite type discrete $G_K$-module.
\begin{defi}
	The Artin L-function of $M$ (twisted by $X$) is defined as the Euler product
	\[
	L_{K,X}(M,s):=L_X(g_\ast M,s)=\prod_{v \in X_{0}} L_v(g_\ast M,s)
	\]
	where $X_0$ are the closed points of $X$, with local factors
	\[
	L_v(g_\ast M,s)=\det(1-N(v)^{-s}\mathrm{Frob}_v|(g_\ast M)_v\otimes\CC)
	\]
	where as before we denote $(-)_v$ the pullback to $v$ and $\mathrm{Frob}_v\in G_v:=\mathrm{Gal}(\kappa(v)^{sep}/\kappa(v))$ is the (arithmetic) Frobenius.
	
	Define $r_{K,X}(M)$ the order at $s=0$ and $L_{K,X}^\ast(M,0)$ the special value at $s=0$.
\end{defi}

\begin{prop}
	If $v$ is a non-singular point, thus corresponding to a place of $K$, then 
	\[
	(g_\ast M)_v=M^{I_v}
	\]
	with its natural $G_v$-action, where $I_v$ is the inertia subgroup of a place $\bar{v}$ of $K^{sep}$ above $v$.
\end{prop}

\begin{proof}
	This is local so we can suppose $X$ is affine and regular. Fix an embedding $K^{sep}\hookrightarrow K_v^{sep}$ ; this determines an extension $\bar{v}$ of $v$ on $K ^{sep}$. Let $\mathcal{O}_{v}^{sh}$ be the strict henselization of $X$ at $v$; then $\mathcal{O}_{v}^{sh}$ is the integral closure of $\mathcal{O}_X(X)$ in the ring of integers of the maximal unramified extension $K_v^{un}$ of the completion of $K$ at $v$. Denote $K_v^{sh}$ the field of fractions. The above embedding identifies $K_v^{sh}$ with the subfield of $K^{sep}$ fixed by the inertia subgroup $I_v$. The formula for stalks of pushforwards then gives
	\[
	(g_\ast M)_v=H^0(K_v^{sh},M)=M^{I_v}
	\]
	endowed with its natural $G_v\simeq D_v/I_v$-action, where $D_v=\{\sigma, \sigma \bar{v}=\bar{v}\}\subset G_K$ is the decomposition group. 
\end{proof}

\begin{rmk}
	\begin{itemize}
		\item[]
		\item If $X$ is regular, the twisted Artin L-function is by the above proposition the usual Artin L-function associated to $S$ :
		\[
		L_{K,X}(M,s)=L_S(M,s)=\prod_{v \notin S}\det(1-N(v)^{-s}\mathrm{Frob}_v|(M\otimes\CC)^{I_v})
		\]
		\item For $M=\ZZ$ we have $g_\ast \ZZ =\ZZ$ and $L_{K,X}(\ZZ,s)=\zeta_X(s)$ is the arithmetic zeta function of $X$.
		\item The singular locus $Z$ of $X$ is finite and if $X$ is regular, it can be completed with finitely many regular points to become proper. Thus the L-function twisted by $X$ differs by a finite number of factors from the ordinary Artin L-function.
	\end{itemize}
\end{rmk}

\begin{cor}\label{special_value_twisted_artin}
Let $f : X \to T$ denote either a quasi-finite morphism $X\to T:=\mathbb{P}^1_{\FF_p}$ (in the function field case), which exists by the projective Noether normalization lemma, or the quasi-finite structural morphism $X\to T:=\Spec(\ZZ)$ (in the number field case); then $f_!$ is well-defined. We have the order and special value formula at $s=0$:
	\begin{align*}
	& r_{K,X}(M)=E_T(f_! g_\ast M)\\
	& L_{K,X}^\ast(M,0)=(-1)^{\mathrm{rank}_\ZZ M} \chi_T(f_!g_\ast M)
	\end{align*}
\end{cor}
\begin{proof}
	The sheaf $g_\ast M$ is $\ZZ$-constructible.
\end{proof}

\begin{cor}\label{special_value_zeta_order}
	Suppose $X$ is affine. We have the order and special value formula at $s=0$ for the arithmetic zeta function of $X$:
	\begin{align*}
	\mathrm{ord}_{s=0}\zeta_X  &= \mathrm{rank}_\ZZ(CH_0(X,1))\\
	\zeta_X^\ast(0) &= - \frac{[\mathrm{CH}_0(X)]R_X}{\omega}
	\end{align*}
	where $R_X$ is the absolute value of the determinant of the regulator pairing
	\[
	H^1_c(X,\ZZ)_\RR \times \mathrm{CH}_0(X,1)_\RR \to \RR
	\]
	after a choice of bases modulo torsion and $\omega$ is the number of roots of unity in $K$.
\end{cor}

\begin{proof}
	We apply cocorollary \ref{special_value_zeta_singular} to $F=\ZZ$. Using $R^q g_\ast \GG_m=0$ for $q>0$ (see \cite[II.1.4]{ADT}), hypercohomology computations give
	\begin{align*}
	H^0(X,\mathbb{G}_X)&=\mathrm{ker}(K^\times \xrightarrow{\sum \mathrm{mult}_v} \bigoplus_{v\in X_{0}} \ZZ)=\mathrm{CH}_0(X,1)\\
	H^1(X,\mathbb{G}_X)&=\mathrm{coker}(K^\times \xrightarrow{\sum \mathrm{mult}_v} \bigoplus_{v\in X_{0}} \ZZ)=\mathrm{CH}_0(X)
	\end{align*}
	Moreover it is clear from the above that $\mathrm{CH}_0(X,1)_{tor}=\mu(K)$. 
	
	From the short exact sequence
	\[
	0 \to \ZZ \to \pi_\ast \ZZ \to \oplus_{v\in Z} i_{v,\ast}\left(\oplus_{\pi(w)=v} \mathrm{ind}^{G_w}_{G_v}\ZZ\right)/\ZZ \to 0,
	\]
	with $\pi : Y\to X$ the normalization, $G_v=\mathrm{Gal}(\kappa(v)^{sep}/\kappa(v))$ and $G_w$ the Galois groups of the residue fields, we find by taking cohomology with compact support that $H^0_c(X,\ZZ)=0$ and $H^1_c(X,\ZZ)$ is torsion-free\footnote{see the computations in the proof of proposition \ref{euler_vs_zeta}}. The regulator $R(\ZZ)$ contains no contribution from the perfect pairing of trivial vector spaces
	\[
	H^0_c(X,\ZZ)_\RR \times CH_0(X)_\RR \to \RR
	\]
	hence we have $R(\ZZ)=R_X$.
\end{proof}

\begin{rmk}
	
	\begin{itemize}
		\item[]
		\item We can compute $\mathrm{rank}_\ZZ \mathrm{CH}_0(X,1)=s+t-1$ where $s=[S]$ and $t=\sum_{v \in Z} ([\pi^{-1}(v)]-1)$ by using the localization sequences associated to the open-closed decompositions $U\hookrightarrow X \hookleftarrow Z$ and $U \hookrightarrow Y \hookleftarrow \pi^{-1} Z$.
		\item Note that the proof above shows that $\mathrm{CH}_0(X)=\Ext^1_X(\ZZ,\GG_X)$ is torsion, and Artin-Verdier duality shows that it is finite type, hence it is finite.
	\end{itemize}
\end{rmk}

\begin{cor}\label{special_value_zeta_proper}
	Suppose $X$ is a proper curve, and denote $\FF_q$ the field of constants of $K$. We have the order and special value formula at $s=0$ for the arithmetic zeta function of $X$:
	\begin{align*}
	\mathrm{ord}_{s=0}\zeta_X  &= \mathrm{rank}_\ZZ(\mathrm{CH}_0(X,1))-1\\
	\zeta_X^\ast(0) &= - \frac{[\mathrm{CH}_0(X)_{tor}]R_X}{\omega \log(q)}
	\end{align*}
	where $R_X$ is the absolute value of the determinant of the regulator pairing
	\[
	H^1(X,\ZZ)_\RR \times \mathrm{CH}_0(X,1)_\RR \to \RR
	\]
	after a choice of bases modulo torsion and $\omega$ is the number of roots of unity in $K$\footnote{Thus $\omega=q-1$}.
\end{cor}

\begin{proof}
As in the previous proof we have $H^0(X,\mathbb{G}_X)=\mathrm{CH}_0(X,1)$, $\mathrm{CH}_0(X,1)_{tor}=\mu(K)$ and $H^1(X,\mathbb{G}_X)=\mathrm{CH}_0(X)$, but this time we have $H^0(X,\ZZ)=\ZZ$. We now want to identify the pairing $H^0(X,\ZZ)_\RR \times \mathrm{CH}_0(X)_\RR \to \RR$ coming from the regulator. For a closed point $w \in Y$ above $v\in X$, denote $f_w:=[\kappa_w:\kappa_v]$ the residual degree. Considering $Y$ as a curve over $\FF_p$, the degree map $\deg_Y:\mathrm{Pic}(Y) \to \ZZ$ has image $f\ZZ$, where $p^f=q$. This easily implies that there is a similarly-defined degree map $\deg_X:\mathrm{CH}_0(X) \to \ZZ$ with image $f\ZZ$.
The snake lemma applied to
\[
\begin{tikzcd}
0 \rar & K^\times/\mathcal{O}_Y(Y)^\times \rar \dar & \bigoplus_{w \in Y_{0}} \ZZ  \rar \dar & \mathrm{Pic}(Y)  \rar \dar & 0\\
0 \rar & K^\times/\mathrm{CH}_0(X,1) \rar & \bigoplus_{v \in X_{0}} \ZZ \rar & \mathrm{CH}_0(X) \rar & 0
\end{tikzcd}
\]
shows that $\mathrm{coker}(\mathrm{Pic}(Y) \to \mathrm{CH}_0(X))=\bigoplus_{v \in Z} \ZZ/m_v\ZZ$, where $m_v=\mathrm{pgcd}\{f_w,~ \pi(w)=v\}$. Since $\mathrm{coker}(\mathrm{Pic}(Y) \to \mathrm{CH}_0(X))=\mathrm{coker}(\mathrm{Pic}^0(Y)\to \mathrm{ker}(\mathrm{deg}_X))$, it follows from the finiteness of $\mathrm{Pic}^0(Y)$ that $\mathrm{ker}(\mathrm{deg}_X)$ is finite and hence equals $\mathrm{CH}_0(X)_{tor}$. It is clear from the constructions that the pairing $H^0(X,\ZZ)_\RR \times \mathrm{CH}_0(X)_\RR \to \RR$ comes from the pairing $H^0(X,\ZZ)\times \mathrm{CH}_0(X) \to \mathrm{CH}_0(X) \xrightarrow{deg} f\ZZ$, hence similarly to the proof of \ref{euler_vs_zeta_proper_curve} we find that its determinant is $f\log(p)=\log(q)$ ; it follows that $R(\ZZ)=R_X/\log(q)$.
\end{proof}

\begin{rmk}
\begin{itemize}
\item[]
\item We can compute $\mathrm{rank}_\ZZ \mathrm{CH}_0(X,1)=t$ where $t=\sum_{v \in Z} ([\pi^{-1}(v)]-1)$ by using the localization sequences associated to the open-closed decompositions $U\hookrightarrow X \hookleftarrow Z$ and $U \hookrightarrow Y \hookleftarrow \pi^{-1} Z$. Thus if $X$ is unibranch we have $R_X=1$.
\item We showed in the proof above that $\mathrm{CH}_0(X)_{tor}$ is the group of classes of $0$-cycles of degree $0$.
\end{itemize}
\end{rmk}

\begin{rmk}
In \cite{Poonen20}, Jordan and Poonen prove an analytic class number formula for the zeta function of a reduced affine finite type $\ZZ$-scheme of pure dimension $1$. They note that even for non-maximal orders in number fields, the formula seems new. Let us compare their results with ours. Define the completed zeta function
\[
\widehat{\zeta}_K(s)
=|d_K|^{s/2}\zeta_Y(s)\prod_{v\in S}L_v(\ZZ,s)
\] where $Y$ is as before the normalization of $X$ in $K$, $d_K$ is the discriminant of $K$, $L_v(\ZZ,s)$ is the usual local factor if $v$ is nonarchimedean and
\begin{align*}
L_v(\ZZ,s)&=2(2\pi)^{s}\Gamma(s), ~v~\text{complex}\\
L_v(\ZZ,s)&=\pi^{\frac {-s} 2}\Gamma(\frac s 2), ~v~\text{real}
\end{align*}
By comparing $\zeta_X$ and $\zeta_Y$, we get the equality
\[
\zeta_X(s)
=\widehat{\zeta}_K(s)|d_K|^{-s/2}\prod_{w\in S}L_w(\ZZ,s)^{-1}\prod_{v\in Z} \left(\frac{L_v(\ZZ,s)}{\prod_{\pi(w)=v}L_w(\ZZ,s)}\right)
\]
Using the functional equation $\widehat{\zeta}_K(s)=\widehat{\zeta}_K(1-s)$, we find the order and special value at $s=0$ :
\begin{align*}
 &\mathrm{ord}_{s=0}\zeta_X  = s+t-1\\
 &\zeta_X^\ast(0) = -\frac{h(\mathcal{O})R(\mathcal{O}))}{\omega(\mathcal{O})[\widetilde{\mathcal{O}}:\mathcal{O}]}\prod_{v\in Z} \frac{\gamma_v}{\prod_{\pi(w)=v}\gamma_w}
\end{align*}
where $Y=\Spec(\widetilde{\mathcal{O}})$, $\gamma_v=\frac{1-N(v)^{-1}}{\log(N(v))}$, $h(\mathcal{O})=[\mathrm{Pic}(X)]$, $R(\mathcal{O})$ is the covolume of the image in $\RR^s$ of $\mathcal{O}^\times$ \textit{via} the classical regulator of $\widetilde{\mathcal{O}}$, and $\omega(\mathcal{O})=\mathcal{O}^{\times}_{tor}$. Our formula at $s=0$ seems to be better in several aspects: it does not suppose $X$ to be reduced nor affine and it is defined in terms of invariants intrisic to $X$.
\end{rmk}
\appendix
\section{Some determinant computations}

Let $M$ be a perfect complex of abelian group or $\RR$-vector spaces. We will denote $H^{ev}M:=\oplus_{i\equiv 0[2]} H^iM$ and $H^{od}M:=\oplus_{i\equiv 1[2]} H^iM$. We fix determinant functors $\det_\ZZ$ and $\det_\RR$ on the derived categories of perfect complexes of abelian groups/$\RR$-vector spaces extending the usual one on finite type projective modules (see section \ref{Construction}).

\begin{lem}
	Let M be a perfect complex of abelian groups with finite cohomology groups. Then
	\[
	\mathrm{det}_\ZZ M \subset (\det_\ZZ M)_\RR \simeq \det_\RR M_\RR \simeq \det_\RR 0 = \RR
	\]
	corresponds to the lattice 
	\[
	\frac 1 {\prod_i [H^i(M)]^{(-1)^i}} \ZZ\subset \RR 
	\]
\end{lem}

\begin{proof}
	We have an isomorphism of determinants\footnote{\cite[Section 5, in particular prop 5.5 and its proof]{Breuning11} and \cite[prop 3.4 and its proof]{Breuning08}} $\det_\ZZ M \simeq \det_\ZZ H^\ast M\simeq \det_\ZZ H^{ev}M \otimes (\det_\ZZ H^{od}M)^{-1}$ compatible with base change, so we may assume $M$ to be concentrated in degree $0$. Since $\det_\ZZ(A\oplus B)\simeq \det_\ZZ(A)\otimes\det_\ZZ(B)$, we further reduce to $M=\ZZ/m\ZZ$. Consider the short exact sequence 
	\[0 \to \ZZ_0 \xrightarrow{m} \ZZ_1 \to M \to 0
	\] where $\ZZ_0=\ZZ_1=\ZZ$. It induces an isomorphism between the determinants $i: \ZZ_0 \otimes \det_\ZZ M \to \ZZ_1$ ; furthermore the induced morphism after base change is $i_\RR : \RR_0 \otimes_\RR \RR=\RR_0 \to \RR_1$, given by $x\mapsto mx$ since it is obtained from the short exact sequence of $\RR$-vector spaces $0 \to \RR_0 \xrightarrow{m} \RR_1 \to 0 \to 0$. Consider now the following commutative diagram
	\[\begin{tikzcd}
		\ZZ_0^{-1}\otimes \ZZ_1 \dar & \ZZ_0^{-1}\otimes \ZZ_0 \otimes \det_\ZZ M \arrow{l}{1\otimes i}[swap]{\simeq} \arrow{r}{ev\otimes 1}[swap]{\simeq}\dar & \det_\ZZ M \dar\\
		\RR_0^{-1}\otimes_\RR \RR_1 & \RR_0^{-1}\otimes_\RR \RR_0 \otimes_\RR \RR \arrow{l}{1\otimes i_\RR=1\otimes m}[swap]{\simeq} \arrow{r}{ev\otimes 1}[swap]{\simeq} & \RR
	\end{tikzcd}\]
	where $ev(\varphi\otimes x)=\varphi(x)$. The integral basis $1^\ast\otimes 1$ of $\ZZ_0^{-1}\otimes \ZZ_1$ is sent under the above maps to $\frac{1}{m}\in \RR$.
\end{proof}

\begin{lem}
	Let $A$ and $B$ be free abelian groups of rank $d$ and $\phi :A_\RR \to B_\RR$ be an isomorphism. Then
	\[
	\det_\ZZ A \otimes (\det_\ZZ B)^{-1}\subset (\det_\ZZ A)_\RR \otimes_\RR ((\det_\ZZ B)^{-1})_\RR \simeq \left(\det_\RR A_\RR \right)\otimes_\RR \left(\det_\RR B_\RR\right)^{-1} \xrightarrow[1\otimes \det_\RR(\phi)^{t}]{\simeq} \left(\det_\RR A_\RR\right)\otimes \left(\det_\RR A_\RR\right)^{-1} \simeq \RR
	\]
	corresponds to the lattice $\det(\phi)\ZZ\subset \RR$ where $\det(\phi)$ is computed with respect to any choice of integral bases of $A$ and $B$.
\end{lem}

\begin{proof}
	Fix bases $(e_1,\ldots,e_d)$ and $(f_1,\ldots, f_d)$ of $A$ and $B$ respectively. Then by definition $\det_\RR(\phi)$ maps $e_1\wedge \cdots \wedge e_d$ to $\det(\phi)f_1\wedge\cdots \wedge f_d$, where $\det(\phi)$ is computed in the mentioned bases. Now the lemma follows from the fact that $\left(\det_\RR A_\RR \right)\otimes_\RR \left(\det_\RR B_\RR\right)^{-1} \to \left(\det_\RR A_\RR\right)\otimes \left(\det_\RR A_\RR\right)^{-1}$ maps the canonical basis element $e_1\wedge \cdots \wedge e_d \otimes (f_1\wedge \cdots \wedge f_d)^\ast$ to $e_1\wedge \cdots \wedge e_d \otimes \det(\phi) (e_1\wedge \cdots \wedge e_d)^\ast$.
\end{proof}

\begin{prop}
	Let $M$ be a perfect complex of abelian groups and let $\phi : H^{ev}(M_\RR) \xrightarrow{\simeq} H^{od}(M_\RR)$ be a "trivialisation" of $M_\RR$. Then
	\[
	\det_\ZZ M \subset (\det_\ZZ M)_\RR \simeq \det_\RR M_\RR \simeq \left(\det_\RR H^{ev}M_\RR \right)\otimes_\RR \left(\det_\RR H^{od} M_\RR\right)^{-1} \xrightarrow[1\otimes_\RR \det_\RR(\phi)^t]{\simeq} \left(\det_\RR H^{ev}M_\RR\right) \otimes_\RR \left(\det_\RR H^{ev} M_\RR\right)^{-1} \simeq \RR
	\]
	corresponds to the lattice 
	\[
	\frac{\det(\phi)}{\prod_i [H^i(M)_{tor}]^{(-1)^i}}\ZZ \subset \RR
	\] where $\det(\phi)$ is computed with respect to any integral bases of $H^{ev}(M)/tor$ and $H^{od}(M)/tor$.
\end{prop}

\begin{proof}
	We have an isomorphism of determinants $\det_\ZZ M \simeq \det_\ZZ H^\ast M\simeq \det_\ZZ H^{ev}M \otimes (\det_\ZZ H^{od}M)^{-1}$ compatible with base change. Using the short exact sequence $0 \to A_{tor} \to A \to A/tor \to 0$ for $A=H^{ev}(M),~H^{od}(M)$, we reduce by functoriality of the determinant to the torsion case and to the free case, which are the two previous lemmas.
\end{proof}

\begin{lem}\label{formula_determinant_acyclic_complexes}
Let $A^\bullet$, $B^\bullet$ be two bounded acyclic complexes of finite type abelian groups, and suppose there is an isomorphism $\phi : B_\RR \xrightarrow{\simeq} \Hom(A_\RR,\RR)[-1]$. Then
\[
\frac{\prod_i [B^i_{tor}]^{(-1)^i}}{\det(\phi)\prod_i[A^i_{tor}]^{(-1)^i}}=1
\]
where $\det(\phi)$ is defined as the alternated product $\prod_i \det(\phi^i : B_\RR^i \to \Hom(A_\RR^{1-i},\RR))^{(-1)^i}$ with determinants computed in bases modulo torsion.
\end{lem}

\begin{proof}
Consider the line $\Delta := \det_\ZZ A \otimes (\det_\ZZ B)^{-1}$. It has a naive trivialisation
\[\det_\ZZ A \otimes (\det_\ZZ B)^{-1} \xrightarrow{\det(0)\otimes (\det(0)^t)^{-1}} \ZZ \otimes \ZZ^{-1} \to \ZZ\]
This induces the naive trivialisation of $\Delta_\RR$, under which the embedding $\Delta \hookrightarrow \Delta_\RR$ corresponds to the natural embedding $\ZZ \hookrightarrow \RR$. Moreover, the isomorphism $\phi$ induces another trivialisation of $\Delta_\RR$:
\[
\det_\RR A_\RR \otimes (\det_\RR B_\RR)^{-1} \xrightarrow{\mathrm{id} \otimes (\det(\phi)^t)^{-1} } \det_\RR A_\RR \otimes (\det_\RR A_\RR)^{-1} \to \RR
\]
under which the embedding $\Delta \hookrightarrow \Delta_\RR$ corresponds by the previous lemmas to 
\[
\frac{\prod_i [B^i_{tor}]^{(-1)^i}}{\det(\phi)\prod_i[A^i_{tor}]^{(-1)^i}}\ZZ \hookrightarrow \RR
\]
Let us show that both trivialisations are compatible. We claim that the following diagram commutes
\[
\begin{tikzcd}[column sep=huge]
\det_\RR A_\RR \otimes (\det_\RR B_\RR)^{-1} \arrow[r, "\mathrm{id} \otimes (\det(\phi)^t)^{-1}"] \arrow[d, "\det(0)\otimes (\det(0)^t)^{-1}"'] & \det_\RR A_\RR \otimes (\det_\RR A_\RR)^{-1} \arrow[d, "nat"] \arrow[ld, "\det(0) \otimes (\det(0)^t)^{-1}" description] \\
\RR \otimes \RR^{-1} \arrow[r, "nat"']                                                                                                          & \RR                                                                                                                      
\end{tikzcd}
\]
The commutativity of the lower triangle is immediate, and the commutativity of the upper triangle comes from the commutative diagram
\[
\begin{tikzcd}
B_\RR \arrow{r}{\phi} \arrow{dr}{0} & \Hom(A_\RR,\RR)[-1] \arrow{d}{0}\\
& 0
\end{tikzcd}
\]
The compatibility of the trivialisations implies that $\alpha=1$, and the formula follows.
\end{proof}

\printbibliography

\end{document}